\documentclass{amsart}
\usepackage{lineno}
\usepackage[margin=1in]{geometry}
\usepackage[latin1]{inputenc}
\usepackage{amsfonts}
\usepackage{enumitem}
\usepackage{color}
\usepackage{bm}
\usepackage[spanish,french,english]{babel}
\usepackage{amssymb,amsmath,amsthm}

\newtheorem{definition}{Definition}[section]
\newtheorem{theorem}{Theorem}[section]
\newtheorem{proposition}{Proposition}[section]
\newtheorem{lemma}{Lemma}[section]
\newtheorem{cor}{Corollary}[section]

\catcode`\@=11
\makeatletter

\@addtoreset{equation}{section}
\numberwithin{equation}{section}

\usepackage{hyperref}

\begin{document}
\title[Second-order regularity for parabolic equations]{Optimal global second-order regularity and improved integrability for parabolic equations with variable growth}



\author{Rakesh Arora}
\address{Department of Mathematical Sciences, Indian Institute of Technology (IIT-BHU), Varanasi, 221005, Uttar Pradesh, India}
\email{rakesh.mat@itbhu.ac.in} 
\thanks{The first author acknowledges the support of the SERB Research Grant SRG/2023/000308, India, and Seed grant IIT(BHU)/DMS/2023-24/493.}

\author{Sergey Shmarev}
\address{Department of Mathematics, University of Oviedo, c/Federico Garc\'{i}a Lorca, Oviedo, 33007, Asturias, Spain}
\email[Corresponding author]{shmarev@uniovi.es}
\thanks{The second author acknowledges the support of the Research Grant MCI-21-PID2020-116287GB-I00, Spain}

\today



\maketitle


\begin{abstract}
We consider the homogeneous Dirichlet problem for the parabolic equation

\[
u_t- \operatorname{div} \left(|\nabla u|^{p(x,t)-2} \nabla u\right)= f(x,t) + F(x,t, u, \nabla u)
\]
in the cylinder $Q_T:=\Omega\times (0,T)$, where $\Omega\subset \mathbb{R}^N$,
$N\geq 2$, is a $C^{2}$-smooth or convex bounded domain. It is assumed that
$p\in C^{0,1}(\overline{Q}_T)$ is a given function, and that the nonlinear
source $F(x,t,s, \xi)$ has a proper power growth with respect to $s$ and
$\xi$. It is shown that if $p(x,t)>\frac{2(N+1)}{N+2}$, $f\in L^2(Q_T)$,
$|\nabla u_0|^{p(x,0)}\in L^1(\Omega)$, then the problem has a solution
$u\in C^0([0,T];L^2(\Omega))$ with $|\nabla u|^{p(x,t)} \in
L^{\infty}(0,T;L^1(\Omega))$, $u_t\in L^2(Q_T)$, obtained as the limit of
solutions to the regularized problems in the parabolic H\"older space. The
solution possesses the following global regularity properties:
\[
\begin{split}
& |\nabla u|^{2(p(x,t)-1)+r}\in L^1(Q_T)\quad \text{for any $0 < r <
\frac{4}{N+2}$},
\qquad
|\nabla u|^{p(x,t)-2} \nabla u \in W^{1,2}(Q_T)^N.
\end{split}
\]
\end{abstract}

\section{Introduction}
We consider the homogeneous Dirichlet problem
\begin{equation}
\label{eq:main}
\begin{split}
& u_t-\operatorname{div}\left(|\nabla u|^{p(z)-2}\nabla
u\right)=f(z)+F(z,u,\nabla u)\quad \text{in $Q_T$},
\\
& \text{$u=0$ on $\partial\Omega\times (0,T)$},
\\
& \text{$u(x,0)=u_0(x)$ in $\Omega$},
\end{split}
\end{equation}
where $Q_T=\Omega\times (0,T)$ is the cylinder of a finite height $T$,
$\Omega\subset \mathbb{R}^N$ is a bounded domain with the $C^2$-smooth or
convex boundary $\partial\Omega$. Here and throughout the text we denote by
$z=(x,t)$ the points of the cylinder $Q_T$. The nonlinear source $F$ has the
form
\begin{equation}
\label{eq:source}
F(z,u,\nabla u)=a(z)|u|^{q(z)-2}u + |\nabla u|^{s(z)-2}(\vec c(z),\nabla u).
\end{equation}
In \eqref{eq:main}, \eqref{eq:source}, the exponents and coefficients $a$,
$\vec c$, $p$, $q$, $s$ are given functions whose properties will be specified
later.

The first objective of this work is to establish the global second-order
spatial differentiability of solutions to problem \eqref{eq:main}. We
prove that if $u_0\in W^{1,p(\cdot,0)}_0(\Omega)$ (the definition of the
variable Sobolev spaces is given in Section \ref{sec:spaces-results}), $f\in
L^2(Q_T)$, and the nonlinear source $F(z,u,\nabla u)$ has a proper growth with
respect to the second and the third arguments, then

\begin{equation}
\label{eq:intr-1}
\begin{split}
& u_t\in L^2(Q_T), \qquad \mathcal{F}(z,\nabla u)\equiv |\nabla u|^{p(z)-2}\nabla u\in W^{1,2}(Q_T)^N,
\\
& \|u_t\|_{L^2(Q_T)}+\|\mathcal{F}(z,\nabla u)\|_{W^{1,2}(Q_T)}\leq C
\end{split}
\end{equation}
with a constant depending only on the problem data, i.e., on
$\|u_0\|_{W^{1,p(\cdot,0)}(\Omega)}$, $\|f\|_{L^2(Q_T)}$, and the nonlinear
structure of the equation. For equations with nonstandard
$p(z)$-growth, this estimate furnishes a natural counterpart of the  classical
inclusion $u \in L^2(0,T; W_0^{2,2}(\Omega))$ for the solutions of
the heat equation, see, e.g., \cite[Ch.3, Sec.6]{LSU}.

In the case of constant $p$, the second-order regularity of solutions
to equations or systems of parabolic equations of the type
\begin{equation}
\label{eq:main-p}
u_t=\Delta_p u+f,\qquad p>1,
\end{equation}
were studied by many authors. We refer to \cite{Cianchi-Maz'ya-2019-1} for
the global estimates \eqref{eq:intr-1} and to
\cite{Feng-Parviainen-Sarsa-2022,Feng-Parviainen-Sarsa-2023,DeFilippis-2020,
Bogelein-2022, Duzaar-Mignione-Steffen-2011, Scheven-2010,
Seregin-Acerbi-Mignione-2004} for local results, see also references therein.
Paper \cite{Feng-Parviainen-Sarsa-2022} deals with the
homogeneous equation \eqref{eq:main-p}. In \cite{Feng-Parviainen-Sarsa-2023}, the local second-order regularity is proven for the viscosity solutions of a more general equation

\begin{equation}
\label{eq:p-normalized}
u_t-|\nabla u|^{\gamma}\left(\Delta u +(p-2)|\nabla u|^{-2}\sum_{i,j=1}^ND^2_{ij}u D_iu D_ju\right)=0
\end{equation}
under some restrictions on the values of the constant exponents $p$ and $\gamma$.  In the special case $\gamma=p-2$ equation \eqref{eq:p-normalized} coincides with \eqref{eq:main-p} with $f=0.$

The results on the second-order spatial regularity of solutions to
\eqref{eq:main-p} are customarily formulated in terms of inclusions

\begin{equation}
\label{eq:reg-prim}
|\nabla u|^{\sigma}\nabla u\in W^{1,2}(Q_T)^N \qquad \text{or
$W^{1,2}_{loc}(Q_T)^N$}
\end{equation}
with some $\sigma$ depending on $p$. In \cite{Cianchi-Maz'ya-2019-1},
inclusion \eqref{eq:reg-prim} with $\sigma=p-2$ is proven for the approximable
solutions such that the solution and its gradient can be obtained as the pointwise
limits of solutions and gradients of solutions to problems \eqref{eq:main-p}
with regularized fluxes and smooth data. The proof in
\cite{Cianchi-Maz'ya-2019-1} crucially depends upon a differential inequality,
which links the Hessian  matrix and the Laplacian of a smooth function and
furnishes a lower bound for the square of the $p$-Laplacian - see \cite[Lemma
3.1]{Cianchi-Maz'ya-2018}.

As it is mentioned in \cite[Remark 2.3]{Cianchi-Maz'ya-2019-1}, the global
estimate \eqref{eq:intr-1} with $\sigma=p-2$ is sharp for the solutions to
problem \eqref{eq:main-p} because in the case $u_0=0$ the norms in
\eqref{eq:intr-1} are bounded below and above by $C_i\|f\|_{L^2(Q_T)}$ with
$C_i=const$.

It is shown in \cite{Feng-Parviainen-Sarsa-2022} that for the weak solutions of equation \eqref{eq:main-p} with $p>1$ $D_i\left(|\nabla u|^{\frac{p-2+s}{2}}D_ju\right)\in L^2_{loc}(Q_T)$ with any $s>-1$. Moreover, a counterexample shows that the inequality $s>-1$ is sharp, see \cite[Remark 2.4]{Feng-Parviainen-Sarsa-2022}.

The inclusion and estimate \eqref{eq:intr-1} proven in our present work are optimal in the following sense. On the one hand, for the solution of
problem \eqref{eq:main} with $F=0$, $u_0=0$ and $f\in L^2(Q_T)$ estimate
\eqref{eq:intr-1} holds with a constant $C$ depending on $\|f\|_{L^2(Q_T)}$.
On the other hand, by virtue of equation \eqref{eq:main}, it is necessary that
$f\in L^2(Q_T)$ for every solution satisfying \eqref{eq:intr-1}.

There is an immense literature on the second-order regularity of solutions to
equations and systems of elliptic nonlinear equations with constant structure.
For relevant results and an exhaustive review of the available literature, we
refer to \cite{Cianchi-Maz'ya-2018, Cianchi-Maz'ya-2019, BCDM-2023}. Results
on the local second-order regularity of solutions to elliptic equations with
$p(x)$-Laplacian can be found in \cite{Wang-2023, Challal-2011}.

Coming back to equation \eqref{eq:main} with the variable exponent $p(z)$, the global
second-order spatial regularity is proven in \cite{Arora-Shm-2021} for
$F\equiv 0$, and in \cite{Arora-Shm-ANONA-2023,Arora-Shm-RACSAM-2023} for
equations with a double-phase flux of variable growth and nonlinear sources.
It is shown in \cite{Arora-Shm-2021} that
\begin{equation}
\label{eq:intr-2}
|\nabla u|^{\frac{p(z)-2}{2}}\nabla u\in W^{1,2}(Q_T)^N,
\end{equation}
provided that $f\in L^2(0,T;W^{1,2}_0(\Omega))$ and $|\nabla
u_0|^{\max\{2,p(x,0)\}}\in L^1(\Omega)$. A specific feature of
equation \eqref{eq:main} with variable structure is that differentiation of
the flux creates first-order residual terms of higher growth, which cannot
be controlled through the principal energy inequality. Nonetheless, these
terms can be estimated  due to the property of global higher integrability of
the gradient  -  see \cite[Theorem 2.3]{Arora-Shm-RACSAM-2023}:
\begin{equation}
\label{eq:high-integr}
|\nabla u|^{p(z)+\delta}\in L^1(Q_T)\quad
\text{with some $\delta>0$}.
\end{equation}
For the solutions of equation \eqref{eq:main} with the regularized flux
\[
\mathcal{F}_{\epsilon}(z,\nabla u)\equiv (\epsilon^2+|\nabla
u|^2)^{\frac{p(z)-2}{2}}\nabla u,\qquad \epsilon>0,
\]
the global property \eqref{eq:high-integr} stems from the following
interpolation inequality: if $\partial\Omega\in C^2$, then for every $v\in
C^2(\overline \Omega)\cap C^3(\Omega)$, $\delta>0$, an exponent $q \in
C^{0,1}(\overline{\Omega})$ satisfying $q(x)>\frac{2N}{N+2}$, and any $r\in
\left(0,\frac{4}{N+2}\right)$,
\begin{equation}
\label{eq:intr-3}
\int_{\Omega}(\epsilon^2+|\nabla v|^2)^{\frac{q(x)-2+r}{2}}|\nabla v|^2\,dx\leq \delta
\int_{\Omega}(\epsilon^2+|\nabla v|^2)^{\frac{q(x)-2}{2}}\sum_{i,j=1}^N\left(D^2_{ij}v\right)^2\,dx+C
\end{equation}
with a constant $C$ depending on $\delta$, $N$, $r$, $q$, and
$\|v\|_{L^2(\Omega)}$. Inequality \eqref{eq:intr-3} was first proven in
\cite{Arora-Shm-2021} for $r$ in the interval
$\left(0,\frac{4p^-}{p^-(N+2)+2N}\right)$ with $p^-=\min p$ and then refined
in \cite{Arora-Shm-RACSAM-2023} to the form used in the present work.

A solution to problem \eqref{eq:main} with the regularized flux
$\mathcal{F}_{\epsilon}(z,\nabla u)$ was constructed in \cite{Arora-Shm-2021}
as the limit of the sequence of finite-dimensional Galerkin's approximations
in the basis composed of the eigenfunctions of the Dirichlet problem for the
Laplace operator. The regularity of eigenfunctions is defined by the
regularity of the boundary $\partial\Omega$, and the sequence of
finite-dimensional approximations is a sequence of smooth functions that
converges to the solution in a proper variable Sobolev space. Inclusion
\eqref{eq:intr-2} follows then from the uniform estimates on the regularized
fluxes.  The same scheme of arguments was used in \cite{Arora-Shm-ANONA-2023,
Arora-Shm-RACSAM-2023} to prove the global second-order differentiability of
solutions to equations with double-phase fluxes of variable growth. A drawback
of this approach is that the a priori estimates for the approximations
require, in general, an extra regularity of the free term $f$ and the initial
datum $u_0$.

For the solutions of the singular equation \eqref{eq:main-p} with variable
$p(z)\in \left(\frac{2N}{N+2},2\right]$, the second-order regularity was
studied in \cite{Ant-Shm-AA-2017,Ant-Shm-NA-2020}, see also
\cite{Scheven-2010} for the case of constant $p$. In this special range of
values of $p(z)$, the inclusions $D^2_{ij}u\in L^2(\Omega\times (s,T))$ with
any $s>0$ and $|D^2_{ij}u|^{p(z)}\in L^{1}(Q_T)$ can be achieved without
higher integrability of the gradient \eqref{eq:high-integr}.

In the present work, we prove the inclusion and estimate \eqref{eq:intr-1} under the assumptions $f\in L^2(Q_T)$, $|\nabla u_0|^{p(x,0)}\in L^1(\Omega)$ and $\min p(z)>\frac{2(N+1)}{N+2}$. The proof is based on a combination of \eqref{eq:intr-3} with the inequality
\begin{equation}
\label{eq:intr-4}
\int_\Omega (\epsilon^2+|\nabla v|^2)^{p(z)-2}\sum_{i,j=1}^{N}\left(D^2_{ij}v\right)^2\,dx\leq C_1\int_{\Omega}\left(\operatorname{div}\mathcal{F}_\epsilon(z,\nabla v)\right)^2\,dx +C_2
\end{equation}
where the constants $C_i$ depend on $\max p$, $\min p$, $\sup |\nabla p|$,
$\|v\|_{L^2(\Omega)}$ and $\partial\Omega$. An immediate
by-product of \eqref{eq:intr-3}-\eqref{eq:intr-4} is the inequality
\begin{equation}
\label{eq:intr-5}
\int_{Q_T}|\nabla u|^{2(p(z)-1)+r}\,dz\leq C\qquad \text{for any $r\in \left(0,\dfrac{4}{N+2}\right)$}
\end{equation}
with a constant $C$ depending only on the data and $r$. Inequality \eqref{eq:intr-5} improves \eqref{eq:high-integr} for $p(z)>2$.

Inequality \eqref{eq:intr-4} holds under the same
assumptions on $v$ and $\partial\Omega$ as \eqref{eq:intr-3}. However,
unlike \eqref{eq:intr-3} inequality \eqref{eq:intr-4} cannot be derived for
Galerkin's approximations, which makes necessary another scheme of approximation of the solution to \eqref{eq:main}.

This is the second objective of our work: to prove that a solution to the
regularized problem \eqref{eq:main} can be obtained as the limit of classical
solutions to the problem with the regularized flux $\mathcal{F}_{\epsilon}$
and the data. Such an approach to construction of a solution - {\sl an
approximable solution} in the terminology of \cite{Cianchi-Maz'ya-2019-1} - is
traditional for problems with constant nonlinearity. To the best of our
knowledge, it is new for problems with the $p(z)$-Laplace operator.  It is an alternative to the approaches based on the finite-dimensional Galerkin's
approximations or regularization of the flux by $|\nabla u|^{p(z)-2}\nabla u + \epsilon |\nabla u|^{p^+-2}\nabla u$, see \cite{Ant-Shm-2009,Al-Zhi-2010,Die-Nag-Ruz-2012} for the existence of
weak solutions to the Dirichlet problem for equation \eqref{eq:main-p} with
the exponent $p(z)$ depending on $t$ and the data $f\in L^2(Q_T)$, $u_0\in L^{2}(\Omega)$.

A classical solution to the problem with the regularized data is obtained as a
fixed point of a mapping generated by a linear problem associated with
\eqref{eq:main}. The existence of a fixed point follows from the
Leray-Schauder principle. The cornerstone of the proof is the recent result of
\cite{Ding-Zhang-Zhou-2020} about the H\"older continuity of the gradients of
weak solutions to equation \eqref{eq:main} with variable nonlinearity and
nonlinear sources of general form. See also \cite{Bogelein-Duzaar-2011,OK-2018} for results on the local H\"older-continuity of the spatial gradient of solutions to the evolution $p(x,t)$ and $p(t)$-Laplacian equation and systems without sources.

The main results of the work are summarized as follows. Assume that
$\partial\Omega \in C^{2+\alpha}$, $p\in C^{0,1}(\overline{Q}_T)$ with $\min
p(z)>\frac{2(N+1)}{N+2}$, the
nonlinear source $F(z,u,\nabla u)$ is subject to proper power growth conditions
in the second and third arguments, and the coefficients and exponents of
nonlinearity of $F$ belong to $C^0(\overline{Q}_T)$.

\begin{itemize}
\item For every $u_0\in W^{1,p(\cdot,0)}_0(\Omega)$, $f\in L^2(Q_T)$ a
solution of problem  \eqref{eq:main} can be obtained as the limit of the
a sequence of classical solutions to the regularized problems, which belong to
the H\"older space $C^{2+\gamma,1+\frac{\gamma}{2}}(\overline{Q}_T)$ with some
$\gamma\in (0,1)$. \item The constructed solution has the properties
\[
u\in C^0([0,T];L^2(\Omega)),\quad u_t\in L^2(Q_T), \quad |\nabla u|^{p(z)}\in L^\infty(0,T;L^1(\Omega)).
\]
\item The solution has the second-order spatial regularity in the sense
\eqref{eq:intr-1} and possesses the property of global higher integrability of the
gradient \eqref{eq:intr-5}. \item All results hold true for the convex domains
$\Omega$ without assumptions on the regularity of $\partial\Omega$, and for
$C^2$-domains.
\end{itemize}

In the proofs, we do not distinguish between the cases $p(z)\geq 2$ and
$p(z)\leq 2$. All conclusions hold true for the variable exponent $p(z)$ which
is allowed to vary within the interval
$\left(\left.\frac{2(N+1)}{N+2},p^+\right]\right.$.

 \section{Function spaces, assumptions, and main results}
\label{sec:spaces-results}

\subsection{Variable Lebesgue and Sobolev spaces}
To formulate the results we have to introduce the function spaces the solution of problem \eqref{eq:main} belongs to. We collect the most necessary facts from the theory of these spaces, for a detailed insight we refer to the monograph \cite{DHHR-2011}.

Let $\Omega$ be a bounded domain with the Lipschitz continuous boundary
$\partial \Omega$ and $p: \Omega \to [p^-, p^+] \subset (1,
\infty)$ be a measurable function. Let us define the functional (the modular)

\[
A_{p(\cdot)}(f)=
\int_{\Omega} |f(x)|^{p(x)} \,dx.
\]
The set
\[
L^{p(\cdot)}(\Omega) = \{f:\Omega
\to \mathbb{R}:\text{$f$ is measurable on $\Omega$}, A_{p(\cdot)}(f) <
\infty\}
\]
equipped with the Luxemburg norm
\[
\|f\|_{p(\cdot),\Omega}= \inf \left\{\lambda>0 :
A_{p(\cdot)}\left(\dfrac{f}{\lambda}\right) \leq 1\right\}
\]
is a reflexive and separable Banach space and $C_0^\infty(\Omega)$ (the set of smooth functions with compact support) is dense in $L^{p(\cdot)}(\Omega)$. The modular $A_{p(\cdot)}(f)$ is lower semicontinuous. By the definition of the norm

\begin{equation}
\label{i0}
\min\{\|f\|^{p^-}_{p(\cdot), \Omega}, \|f\|^{p^+}_{p(\cdot), \Omega}\}
\leq A_{p(\cdot)}(f) \leq \max\{\|f\|^{p^-}_{p(\cdot), \Omega}, \|f\|^{p^+}_{p(\cdot), \Omega}\}.
\end{equation}
The dual of $L^{p(\cdot)}(\Omega)$ is the space $L^{p'(\cdot)}(\Omega)$ with the conjugate exponent  $p'=\dfrac{p}{p-1}$. For $f \in L^{p(\cdot)}(\Omega)$ and $g \in L^{p'(\cdot)}(\Omega)$, the generalized
H\"older inequality holds:

\begin{equation}
\label{eq:gen-Holder}
\int_{\Omega} |fg| \leq
\left(\frac{1}{p^-} + \frac{1}{(p')^-} \right) \|f\|_{p(\cdot),
\Omega} \|g\|_{p'(.), \Omega} \leq 2 \|f\|_{p(\cdot), \Omega}
\|g\|_{p'(\cdot), \Omega}.
\end{equation}
Let $p_1, p_2$ be two bounded
measurable functions in $\Omega$ such that $1<p_1(x) \leq p_2(x)$ a.e. in $\Omega$.  Then $L^{p_1(\cdot)}(\Omega)$ is continuously
embedded in $L^{p_2(\cdot)}(\Omega)$ and

\[
\|u\|_{p_1(\cdot), \Omega} \leq C(|\Omega|,
p_1^\pm, p_2^\pm) \|u\|_{p_2(\cdot), \Omega}, \qquad \forall \,u \in
L^{p_2(\cdot)}(\Omega).
\]
The variable Sobolev space $W_0^{1,p(\cdot)}(\Omega)$ is defined as the set of functions

\[
W_0^{1,p(\cdot)}(\Omega)= \{u: \Omega \mapsto
\mathbb{R}\ |\  u \in L^{p(\cdot)}(\Omega) \cap W_0^{1,1}(\Omega),
|\nabla u| \in L^{p(\cdot)}(\Omega) \}
\]
equipped with the norm

\[
\|u\|_{W_0^{1,p(\cdot)}(\Omega)}= \|u\|_{p(\cdot),\Omega} +
\|\nabla u\|_{p(\cdot), \Omega}.
\]
Under the assumptions $p\in C^{0}(\overline \Omega)$ and $\partial\Omega\in Lip$, the
Poincar\'{e} inequality holds: there is a finite constant $C>0$ such that

\[
\|v\|_{p(\cdot),\Omega}\leq C\|\nabla v\|_{p(\cdot),\Omega}\qquad \forall \ v\in W_0^{1,p(\cdot)}(\Omega).
\]
It follows that $\|\nabla v\|_{p(\cdot),\Omega}$ is an equivalent norm of
$W^{1,p(\cdot)}_0(\Omega)$.

It is known that $C_0^{\infty}(\Omega)$ is dense in
$W_0^{1,p(\cdot)}(\Omega)$,
provided $p
\in C_{{\rm log}}(\overline{\Omega})$, {\it i.e.}, $p$ is continuous in $\overline{\Omega}$ with the logarithmic modulus of continuity:

\begin{equation}
\label{eq:log-cont}
|p(x_1)-p(x_2)| \leq \omega(|x_1-x_2|),
\end{equation}
where $\omega(\cdot)$ is a nonnegative function satisfying the condition
\[
\limsup_{\tau \to 0^+} \omega(\tau) \ln
\left(\frac{1}{\tau}\right)= C< \infty.
\]
For the study of parabolic problem \eqref{eq:main}, we need the spaces of functions depending on $(x,t)\in Q_T$:
\[
\begin{split}
 & \mathbf{V}_t(\Omega) = \{u: \Omega \mapsto \mathbb{R}\ |\  u \in L^2(\Omega)
\cap W_0^{1,1}(\Omega), |\nabla u|^{p(x,t)} \in L^{1}(\Omega)
\},\quad t\in (0,T),
\\
& \mathbf{W}_{p(\cdot)}(Q_T) = \{u : (0,T) \mapsto \mathbf{V}_t(\Omega) \ |\ u \in L^2(Q_T), |\nabla
u|^{p(x,t)} \in L^1(Q_T)\}.
\end{split}
\]

\subsection{The parabolic H\"older spaces}
\label{subsec:Holder-spaces}
Given two points $z_i=(x_i,t_i)\in Q_T$,
$i=1,2$, we denote by

\[
\rho(z_1,z_2)=|x_2-x_1|+|t_2-t_1|^{\frac{1}{2}}
\]
the parabolic distance. The space of H\"older-continuous functions
$C^{\alpha,\frac{\alpha}{2}}(\overline{Q}_T)$, $\alpha\in (0,1)$,  is the collection
of all functions with the finite norm
\[
\|u\|_{\alpha,\frac{\alpha}{2},Q_T}=|u|_{0,Q_T}+\left[ u
\right]_{\alpha,\frac{\alpha}{2},Q_T},
\]
where
\[
|u|_{0,Q_T}=\sup_{Q_T}|u|,\qquad \left[ u
\right]_{\alpha,\frac{\alpha}{2},Q_T}= \sup_{z_1,z_2\in Q_T,
z_1\not=z_2}\frac{|u(z_1)-u(z_2)|}{\rho^{\alpha}(z_1,z_2)}.
\]
The spaces $C^{2+\alpha,1+\frac{\alpha}{2}}(\overline{Q}_T)$ and
$C^{1+\alpha,\frac{1+\alpha}{2}}(\overline{Q}_T)$ are the spaces of functions
with the  finite norms

\[
\|u\|_{2+\alpha,1+\frac{\alpha}{2},Q_T}=\sum_{0\leq i+2j\leq
2}\left|D_t^jD_x^{i}u\right|_{0,Q_T} + \sum_{i+2j= 2}\left[ D_t^jD_x^{i}u
\right]_{\alpha,\frac{\alpha}{2},Q_T},
\]

\[
\|u\|_{1+\alpha,\frac{1+\alpha}{2},Q_T}=|u|_{0,Q_T}+\sum_{i=1}^N|D_iu|_{0,Q_T}
+\sum_{i=1}^N\left[ D_iu\right]_{\alpha,\frac{\alpha}{2},Q_T}
+\sup_{t\not=s,x}\dfrac{|u(x,t)-u(x,s)|}{|t-s|^{\frac{1+\alpha}{2}}}.
\]

We will use the known properties of the H\"older
spaces on smooth domains:
\begin{enumerate}
\item for every $v\in C^{1+\alpha,\frac{1+\alpha}{2}}(\overline{Q}_T)$

\[
\|\nabla v\|_{\alpha,\frac{\alpha}{2},Q_T}\leq
C\|v\|_{1+\alpha,\frac{1+\alpha}{2},Q_T};
\]

\item for every $u\in C^{2+\alpha,1+\frac{\alpha}{2}}(\overline{Q}_T)$
\[
\|\nabla u\|_{\alpha,\frac{\alpha}{2},Q_T}\leq
\|u\|_{1+\alpha,\frac{1+\alpha}{2},Q_T} \leq
C\|u\|_{2+\alpha,1+\frac{\alpha}{2},Q_T};
\]

\item the embedding $C^{2+\alpha,1+\frac{\alpha}{2}}(\overline{Q}_T)\subset
C^{1+\gamma,\frac{1+\gamma}{2}}(\overline{Q}_T)$ is compact for any $\gamma\in
(0,1)$;

    \item for every $f,g\in C^{k+\alpha,\frac{k+\alpha}{2}}(\overline Q_T)$, $k=0,1,2$,

    \begin{equation}
    \notag
    \label{eq:prod-holder}
    fg\in C^{k+\alpha,\frac{k+\alpha}{2}}(\overline Q_T),\qquad \|fg\|_{k+\alpha,\frac{k+\alpha}{2},Q_T}
    \leq \|f\|_{k+\alpha,\frac{k+\alpha}{2},Q_T}\|g\|_{k+\alpha,\frac{k+\alpha}{2},Q_T}.
    \end{equation}
\end{enumerate}

By $C^{0,1}(\overline Q_T)$ we denote the space of Lipschitz-continuous functions: a function $\phi$ is Lipschitz-continuous on $\overline Q_T$ if there exists a finite constant $L\geq 0$ such that

\[
|\phi(x,t)-\phi(y,\tau)|\leq Ld\left((x,t),(y,t)\right)\qquad \forall
(x,t),(y,\tau)\in \overline{Q}_T,
\]
where $d(\cdot,\cdot)$ is the Euclidean distance in $\mathbb{R}^{N+1}$.

\subsection{Assumptions and main results}
We assume that the exponents of nonlinearity and the coefficients in equation \eqref{eq:main} satisfy the following conditions:
\[
p\in C^{0,1}(\overline{Q}_T),\quad  q,s,a \in C^{0}(\overline Q_T), \quad \vec c \in C^0(\overline Q_T)^N,
\]
there exist positive constants $p^\pm$, $q^\pm$, $s^\pm$, $\mu$, $a^+$, $c_i^+$ such that $2\mu<p^--1$,
\begin{equation}
\label{eq:p}
\begin{split}
& p:\,\overline{Q}_T\mapsto [p^-,p^+],\qquad \dfrac{2(N+1)}{N+2}<p^-\leq
p^+<\infty,
\end{split}
\end{equation}

\begin{equation}
\label{eq:q}
q:\,\overline{Q}_T\mapsto [q^-,q^+], \quad 1<q^-\leq q^+<\min\left\{p^-,1+p^-\frac{N+2}{2N}\right\}-2\mu,
\end{equation}

\begin{equation}
\label{eq:s}
\begin{split}
& s:\,\overline{Q}_T\mapsto [s^-,s^+], \quad 1<s^-\leq s^+<\infty,\quad s^+\leq p^--2\mu,
\end{split}
\end{equation}

\begin{equation}
\label{eq:data}
\max_{\overline{Q}_T}|a|=a^+,\quad
\max_{\overline{Q}_T}|c_i|=c_i^+.
\end{equation}

\begin{definition}
\label{def:1} A function $u$ is called strong solution of problem
\eqref{eq:main} if

\begin{enumerate}
\item $u\in C^0([0,T];L^2(\Omega))\cap
\mathbf{W}_{p(\cdot)}(Q_T)$, $u_t\in L^2(Q_T)$,

    \item for every $\phi\in \mathbf{W}_{p(\cdot)}(Q_T)$

    \begin{equation}
    \label{eq:def-1}
    \int_{Q_T}\left(u_t\phi+|\nabla u|^{p(z)-2}\nabla u\cdot \nabla \phi-f\phi -F(z,u,\nabla u)\phi\right)\,dz=0,
    \end{equation}

    \item for every $\phi\in C^0(\Omega)$

    \[
    \int_\Omega (u(x,t)-u_0(x))\phi(x)\,dx\to 0\quad \text{as $t\to 0^+$}.
    \]
\end{enumerate}
\end{definition}

\textbf{Notation.} By agreement, we denote $p_0(x)=p(x,0)$ and use the
abbreviations

\[
|v_{xx}|^2=\sum_{i,j=1}^N\left(D_{ij}^2v\right)^2,\qquad \|v\|_{L^s(\omega)}=\|v\|_{s,\omega}.
\]

The notation $\partial \Omega\in C^{2+\alpha}$ or $C^2$ means that for every
$x_0\in \partial\Omega$ there exists a ball $B_R(x_0)$ such that in the local
coordinate-system $\{y_i\}$ centered at $x_0$ with $y_N$ pointing in the
direction of the exterior normal to $\partial\Omega$ at $x_0$, the set
$B_R(x_0)\cap \partial\Omega$ can be represented as the graph of a function
$\phi(y_1,\ldots,y_{N-1})\in C^{2+\alpha}$ or $C^2$.

The symbol $C$ stands for constants that can be evaluated through the known quantities but whose exact values are unimportant. The value of $C$ may vary from line to line even inside the same formula.

\begin{theorem}
\label{th:existence-degenerate} Assume that conditions
\eqref{eq:p}-\eqref{eq:data} are fulfilled and $\partial\Omega\in
C^{2+\alpha}$, $\alpha\in (0,1)$. Then  for every $f \in L^2(Q_T)$ and
$u_{0}\in W_0^{1,p_0(\cdot)}(\Omega)$ problem \eqref{eq:main} has a strong
solution $u$. The solution satisfies the estimate

\begin{equation}
\label{eq:unif-est-degenerate}
\begin{split}
\operatorname{ess}\sup_{(0,T)}\|u\|_{2,\Omega} & +\operatorname{ess}\sup_{(0,T)}\|\nabla u\|_{p(\cdot),\Omega}+\|u_{t}\|_{2,Q_T}
\leq C
\end{split}
\end{equation}
with a constant $C$ depending only on  $N$, $\partial\Omega$, the structural
constants in conditions \eqref{eq:p}-\eqref{eq:data}, $\|\nabla
u_{0}\|_{p_0(\cdot),\Omega}$ and  $\|f\|_{2,Q_T}$. Moreover, the solution
possesses the property of  global higher integrability of the gradient:

\begin{equation}
\label{eq:higher-m} \int_{Q_T}|\nabla u|^{2(p(z)-1)+r}\,dz\leq C\qquad \text{with
any $r\in \left(0,\frac{4}{N+2}\right)$}
\end{equation}
and a constant $C$ depending on $r$ and the same quantities as the constant in \eqref{eq:unif-est-degenerate}.
\end{theorem}
\begin{theorem}
\label{th:reg-degenerate} Let the conditions of Theorem
\ref{th:existence-degenerate} be fulfilled and $u$ be the strong solution of
problem \eqref{eq:main} obtained in Theorem \ref{th:existence-degenerate}.
Then

\[
|\nabla u|^{p(z)-2}\nabla u\in W^{1,2}(Q_T)^N \quad \text{and} \quad \||\nabla u|^{p(z)-2}\nabla u\|_{W^{1,2}(Q_T)}\leq C
\]
with a constant $C$ depending on $N$, the constants in conditions \eqref{eq:p}-\eqref{eq:data}, $\|\nabla u_{0}\|_{p_0(\cdot),\Omega}$ and  $\|f\|_{2,Q_T}$.
\end{theorem}

\begin{theorem}
\label{th:convex-domain} The assertions of Theorems
\ref{th:existence-degenerate}, \ref{th:reg-degenerate} remain true  for the
$C^2$-domains and convex domains $\Omega$ without any assumption on the
regularity of $\partial\Omega$.
\end{theorem}

\subsection{Organization of the work}
Section \ref{sec:regularization} is devoted to studying the problem with the regularized nondegenerate flux $\mathcal{F}_{\epsilon}(z,\nabla u)$ and smooth data in a smooth domain $\Omega$. The exponents of nonlinearity, the coefficients, and the data $f$ and $u_0$ are approximated by sequences of smooth functions converging in the corresponding norms.
A solution to the regularized problem
is sought as the fixed point of the mapping generated by the linearized
problem.  The bulk of the section consists in checking the fulfillment of the
conditions of the Leray-Schauder principle. In the proof, we rely on the recent
results on boundedness and H\"older continuity of the gradient for the weak
solutions to the corresponding nonlinear problem \cite{Ding-Zhang-Zhou-2020}.
The constructed solution belongs to $C^{2+\beta,1+\frac{\beta}{2}}(\overline
Q_T)$ with some $\beta\in (0,1)$, its norm depends on $\epsilon$.

In Section \ref{sec:flux} we derive the integral inequality which allows one to estimate $\|(\epsilon^2+|\nabla u|^2)^{\frac{p-2}{2}}|u_{xx}|\|_{2,\Omega}$ through $\|\operatorname{div}\mathcal{F}_{\epsilon}(z,\nabla u)\|_{2,\Omega}$ and the data. We use first the known integral identity, which holds for smooth functions and domains:

\begin{equation}
\label{eq:integral-main}
\begin{split}
\int_{\Omega} (\operatorname{div}\mathcal{F}_\epsilon(z,\nabla u))^2\,dx  = & - \int_{\partial\Omega}\mathcal{K}(\mathcal{F}_\epsilon(z,\nabla u),\vec \nu)^2
\,dS
\\
&
+\sum_{i,j=1}^N\int_{\Omega}D_{i}\left(\mathcal{F}^{(j)}_\epsilon(z,\nabla u)\right) D_{j}\left(\mathcal{F}^{(i)}_\epsilon(z,\nabla u)\right)\,dx,
\end{split}
\end{equation}
where $\mathcal{K}$ is the second fundamental form of the surface
$\partial\Omega$,   and $\vec \nu$ is the unit exterior normal to
$\partial\Omega$. By straightforward computations, the integrand of the second
integral is represented as the sum of $(\epsilon^2+|\nabla
u|^2)^{p-2}|u_{xx}|^2$ and several residual terms which
appear only if $p$ depends on $x$. After estimating these terms and the
boundary integral, we arrive at \eqref{eq:intr-4}. This inequality together
with another known inequality \eqref{eq:intr-3} imply higher integrability of
the gradient \eqref{eq:intr-5}.

In Section \ref{sec:estimates} uniform a priori estimates for the classical
solutions of the regularized problems are derived. The assumptions on the
growth of $F$ allow one to estimate the nonlinear sources through the elliptic
part of the equation and the lower-order norms of the solutions. In our
assumptions, all terms of the regularized equation are uniformly bounded in
$L^2(Q_T)$.

In Section \ref{sec:proofs} we prove Theorems \ref{th:existence-degenerate},
\ref{th:reg-degenerate}. We prove first the analogs of these theorems for
problem \eqref{eq:main} with the regularized flux $\mathcal{F}_\epsilon$. The
uniform estimates allow us to choose a sequence of solutions to the
regularized problems with smooth data such that the sequence itself, the
derivatives of its terms, and the corresponding nonlinear terms are weakly
convergent. The limits of the nonlinear terms are identified by the
monotonicity of the fluxes and the pointwise convergence of the gradients
together with the higher integrability property \eqref{eq:intr-5}. The
assertions of Theorems \ref{th:existence-degenerate}, \ref{th:reg-degenerate}
follow after passing to the limit $\epsilon\to 0^+$ in the family of solutions
to the regularized problems.

The proof of Theorem \ref{th:convex-domain} is given in Section
\ref{sec:convex-domain}.  For a convex domain, this is done by means of
approximation from the interior by a sequence of smooth convex domains,
application of Theorems \ref{th:existence-degenerate}, \ref{th:reg-degenerate}
in each of these domains, and the choice of a sequence of solutions with
proper convergence properties. For the smooth convex domains, the function
$\mathcal{K}$ in \eqref{eq:integral-main} is non-positive. In this case, the
first term on the right-hand side of \eqref{eq:integral-main} can be omitted,
and the key estimates on the solutions to the problems in the regularized
domains become independent of $\partial\Omega$. In the case of a $C^2$-domain
$\Omega$, the conclusion follows by approximation of $\Omega$ from the
interior by a family of smooth domains whose boundaries are uniformly bounded
in $C^2$.

\section{Regularized problem}
\label{sec:regularization} A solution of the degenerate problem
\eqref{eq:main} is obtained as the limit of the family of solutions to the
following problems with the regularized flux and sources
\begin{equation}
\label{eq:main-reg}
\begin{split}
& u_t-\operatorname{div}\left((\epsilon^2+|\nabla
u|^{2})^{\frac{p(z)-2}{2}}\nabla u\right)= F_\epsilon (z,u,\nabla
u)+f\quad \text{in $Q_T$},
\\
& \text{$u=0$ on $\partial\Omega\times (0,T)$},
\\
& \text{$u(x,0)=u_0(x)$ in $\Omega$},\qquad \epsilon\in (0,1),
\end{split}
\end{equation}
where
\[
F_\epsilon(z,u,\nabla u)=a(\epsilon^2+u^2)^{\frac{q-2}{2}}u + (\epsilon^2+|\nabla u|^2)^{\frac{s-2}{2}}(\vec c,\nabla u),\qquad (z,u,\nabla u)\in Q_T\times \mathbb{R}\times \mathbb{R}^N.
\]
Throughout the text until Section \ref{sec:convex-domain} we
assume that $\partial\Omega\in C^{2+\alpha}$ with some $\alpha\in (0,1)$.

\subsection{Properties of the regularized flux}

Fix $\epsilon>0$ and accept the notation

\begin{equation}
\label{eq:reg-flux}
\mathcal{G}(p,\vec \xi)=
\left(\epsilon^2+|\vec \xi|^2\right)^{\frac{p-2}{2}},\qquad
\mathcal{F}_\epsilon(z,\nabla u)=\mathcal{G}(p(z),\nabla u)\nabla u.
\end{equation}
The function $\mathcal{F}_\epsilon$ with $\epsilon>0$ is the
regularized flux. It is straightforward to check that $\mathcal{F}_\epsilon$
with $\epsilon>0$ possesses the following properties.

\begin{enumerate}
\item \textsl{Monotonicity.} There is a constant $C=C(p^\pm)$ such that for every $\xi,\eta\in \mathbb{R}^N$, $z\in Q_T$

\begin{equation}
\label{eq:mon-reg}
(\mathcal{F}_{\epsilon}(z,\xi)-\mathcal{F}_{\epsilon}(z,\eta),\xi-\eta)\geq C \begin{cases}
|\xi-\eta|^{p(z)} & \text{if $p(z)\geq 2$},
\\
(1+|\xi|^2+|\eta|^2)^{\frac{p(z)-2}{2}}|\xi-\eta|^2 & \text{if $p(z)\in (1,2)$}.
\end{cases}
\end{equation}
For the proof, we refer to \cite[Proposition 3.1]{Arora-Shm-RACSAM-2023}.

\item \textsl{Coercivity.} For every $z\in Q_T$ and $\xi\in \mathbb{R}^N$

\begin{equation}
\label{eq:coercive}
\mathcal{F}_\epsilon(z,\xi)\cdot \xi \geq \begin{cases}
|\xi|^{p(z)} & \text{if $p(z)\geq 2$},
\\
C |\xi|^{p(z)} - 2\epsilon^{p(z)} & \text{if $p(z)\in (1,2)$}.
\end{cases}
\end{equation}

\begin{proof}
If $p\geq 2$, inequality \eqref{eq:coercive} is obvious. Let $p\in (1,2)$. By Young's inequality, for every $\delta\in (0,1)$

\begin{equation}
\label{eq:mon-flux-aux}
\begin{split}
|\xi|^p & = \left(\delta^{\frac{p-2}{p}}|\xi|^2(\epsilon^2 +|\xi|^2)^{\frac{p-2}{2}}\right)^{\frac{p}{2}} \left(\delta (\epsilon^2+|\xi|^2)^{\frac{p}{2}}\right)^{\frac{2-p}{2}}
\\
&
\leq \delta(\epsilon^2+|\xi|^2)^{\frac{p}{2}}+ \delta^{1-\frac{2}{p}} (\epsilon^2+|\xi|^2)^{\frac{p-2}{2}}|\xi|^2.
\end{split}
\end{equation}
If $\epsilon^2\leq |\xi|^2$, then $(\epsilon^2+|\xi|^2)^{\frac{p}{2}}\leq 2^{\frac{p}{2}}|\xi|^p\leq 2|\xi|^p$. The second line of \eqref{eq:coercive} with $C=2\cdot 4^{\frac{2}{p}-1}$ follows upon choosing $\delta=1/4$. If $|\xi|^2< \epsilon^2$, we may choose $\delta=1$:

\[
|\xi|^p\leq 2\epsilon^p + (\epsilon^2+|\xi|^2)^{\frac{p-2}{2}}|\xi|^2.
\]
\end{proof}

\item \textsl{Growth.} There exist constants $C_1$, $C_2$ such that for every $z\in Q_T$ and $\xi \in \mathbb{R}^N$

    \begin{equation}
    \label{eq:growth}
    |\mathcal{F}_{\epsilon}(z,\xi)|\leq C_1|\xi|^{p(z)-1}+C_2.
    \end{equation}
    Inequality \eqref{eq:growth} follows from

    \[
    \begin{split}
    |\mathcal{F}_\epsilon (z,\xi)| & \leq (\epsilon^2+|\xi|^2)^{\frac{p(z)-2}{2}}|\xi|\leq (\epsilon^2+|\xi|^2)^{\frac{p(z)-1}{2}}
    \leq 2^{\frac{p^+-1}{2}} \begin{cases}
    \epsilon^{p(z)-1} & \text{if $|\xi|\leq \epsilon$},
    \\
    |\xi|^{p(z)-1} & \text{if $|\xi|\geq \epsilon$}.
    \end{cases}
    \end{split}
    \]

    \item \textsl{H\"older-continuity.} Let $p\in C^{\beta,\frac{\beta}{2}}(\overline{Q}_T)$, $\beta\in (0,1)$. For every $z=(x,t), s=(y,\tau)$, $z,s\in Q_T$, and every $\xi\in \mathbb{R}^N$

        \begin{equation}
        \label{eq:flux-cont}
        |\mathcal{F}_\epsilon(z,\xi)-\mathcal{F}_{\epsilon}(s,\xi)|\leq C \left(|x-y|^\beta+|t-\tau|^{\frac{\beta}{2}}\right)\left(1+\ln (1+|\xi|^2)\right)\left(1+|\xi|^{r}\right)
        \end{equation}
        with $r=\max\{p(z),p(s)\}$ and an independent of $\xi$ constant $C$.

        \begin{proof}
        Denote by $\mathcal{F}_\epsilon^{(i)}$ the $i$th component of the flux. By the mean value theorem

        \[
        \begin{split}
        \mathcal{F}^{(i)}_\epsilon(z,\xi)-\mathcal{F}^{(i)}_{\epsilon}(s,\xi) & = \xi_i\int_{0}^1\dfrac{d}{d\theta}(\epsilon^2+|\xi|^2)^{\frac{1}{2}(\theta p(z)+(1-\theta)p(s)-2)}\,d\theta
        \\
        & = \xi_i\int_0^1 (\epsilon^2+|\xi|^2)^{\frac{1}{2}(\theta p(z)+(1-\theta)p(s)-2)}\ln (\epsilon^2+|\xi|^2) \,d\theta (p(z)-p(s)),
        \end{split}
        \]
        whence

        \[
        |\mathcal{F}^{(i)}_\epsilon(z,\xi)-\mathcal{F}^{(i)}_{\epsilon}(s,\xi)|\leq |p(z)-p(s)|\int_0^1 (\epsilon^2+|\xi|^2)^{\frac{1}{2}(\theta p(z)+(1-\theta)p(s)-1)}|\ln (\epsilon^2+|\xi|^2)|\,d\theta.
        \]
        For every constant $\mu\in (0,p^--1)$

        \[
        |\ln (\epsilon^2+|\xi|^2)|\leq C_\mu\begin{cases}
        \ln (1+|\xi|^2) & \text{if $\epsilon^2+|\xi|^2\geq 1$},
        \\
        (\epsilon^2+|\xi|^2)^{-\mu} & \text{if $\epsilon^2+|\xi|^2< 1$}.
        \end{cases}
        \]
        Assume, for the sake of definiteness, that $p(z)>p(s)$. If $\epsilon^2+|\xi|^2\geq 1$, then

        \[
        \begin{split}
        |\mathcal{F}_\epsilon(z,\xi)-\mathcal{F}_{\epsilon}(s,\xi)| & \leq C|p(z)-p(s)|(\epsilon^2+|\xi|^2)^{\frac{p(z)-1}{2}}\left(1+\ln(1+|\xi|^2)\right)
        \\
        & \leq C'|p(z)-p(s)|(1+|\xi|^{p(z)-1})\left(1+\ln(1+|\xi|^2)\right).
        \end{split}
        \]
        Otherwise, $\epsilon^2+|\xi|^2<1$

        \[
        \begin{split}
        |\mathcal{F}_\epsilon(z,\xi)-\mathcal{F}_{\epsilon}(s,\xi)| & \leq C_{\mu}|p(z)-p(s)|(\epsilon^2+|\xi|^2)^{\frac{p(s)-1}{2}-\frac{\mu}{2}}
        \\
        & \leq C'_\mu |p(z)-p(s)|\left(1+|\xi|^{p(z)-1}\right) \left(1+\ln(1+|\xi|^2)\right).
        \end{split}
        \]
        Here we have used the inequality
        \[
        \begin{split}
        (\epsilon^2+|\xi|^2)^{\frac{p(s)-1-\mu}{2}} & \leq 2^{\frac{p(s)-1-\mu}{2}}\begin{cases}
        \epsilon^{p(s)-1-\mu} & \text{if $\epsilon\geq |\xi|$},
        \\
        |\xi|^{p(s)-1-\mu} & \text{if $\epsilon< |\xi|$}
        \end{cases}
        \\
        &
        \leq 2^{\frac{p^+-1-\mu}{2}} (\epsilon^{p(s)-1-\mu}+|\xi|^{p(s)-1-\mu})
        \end{split}
        \]
        and Young's inequality.
        \end{proof}
\end{enumerate}

\subsection{Classical solutions of the regularized problem}

Let us approximate the data of problem \eqref{eq:main-reg}:
\begin{equation}
\label{eq:data-reg}
\begin{split}
& u_{0m}\in C_0^{\infty}(\Omega),\qquad \text{$u_{0m}\to u_0$ in
$W^{1,p_0(\cdot)}_0(\Omega)$},
\\
& f_m\in C_0^{\infty}(Q_T),\qquad \text{$f_m\to f$ in $L^2(Q_T)$}.
\end{split}
\end{equation}
We refer to \cite[Sec.4]{Die-Nag-Ruz-2012} for the approximation of $u_0$ in
the variable Sobolev space $W^{1,p_0(\cdot)}_0(\Omega)$, the approximation of
$f$ is standard. The exponents and coefficients are approximated as follows:

\begin{equation}
\label{eq:data-reg-0}
\begin{split}
& p_k\in C^\infty (\overline Q_T), \quad \text{$p_k\nearrow p$ in $C^{0,1}(\overline Q_T)$},
\\
& s_k,q_k\in C^{\infty}(\overline Q_T), \quad \text{$q_k\to q$, $s_k\to s$ in $C^0(\overline Q_T)$},
\\
& a_k,c_{ik}\in C^{\infty}(\overline Q_T), \quad \text{$a_k\to a$, $c_{ik}\to c_i$ in $C^0(\overline Q_T)$}.
\end{split}
\end{equation}
By the McShane-Whitney extension theorem \cite{McShane-1934}
every continuous or Lipschitz-con\-ti\-nu\-ous function can be extended to the
whole space in such a way that the extension preserves the modulus of
continuity. The approximations for $q$, $s$, $a$, $\vec c$ are obtained then
by mollification of the extended functions. To obtain a monotone increasing
sequence $\{p_m\}$ we extend $p$ from $Q_T$ to $\mathbb{R}^{N+1}$ by a
function $P$ with the same Lipschitz constant $L$, take the monotone
increasing sequence $P_m=P-2^{-m}$ and set $p_m=P_m\star \phi_{2^{-2m}}$,
where $\phi$ is the standard mollifier. For all $m$ from some $m_0$ on, the
approximations satisfy the structural conditions \eqref{eq:q}-\eqref{eq:data}
with the same constants except for $2\mu$, which is substituted by $\mu$. From
now on, we use the abbreviate notation for the sequences of approximations
satisfying \eqref{eq:data-reg-0}:
\[
    \textbf{data}_m \equiv  
    (p_m, q_m,  s_m, a_m, \vec c_{m}).
\]
For the chosen $\textbf{data}_m$

\[
\begin{split}
& \|\nabla u_{0m}\|_{p_m(\cdot,0),\Omega}\leq C\|\nabla u_{0m}\|_{p_0(\cdot),\Omega}\leq C'\|\nabla u_0\|_{p_0(\cdot),\Omega},
\\
& \|u_{0n}-u_{0m}\|_{2,\Omega}\leq C\|\nabla (u_{0n}-u_{0m})\|_{p^-,\Omega}\leq C'\|u_{0n}-u_{0m}\|_{W^{1,p_0(\cdot)}_0(\Omega)}\to 0\quad \text{as $m,n\to \infty$}.
\end{split}
\]
Consider problem \eqref{eq:main-reg} with the data $f_m$, $u_{0m}$ and $\textbf{data}_m$,
\begin{equation}
\label{eq:main-reg-1}
\begin{split}
& u_t-\operatorname{div}\left((\epsilon^2+|\nabla
u|^{2})^{\frac{p_m(z)-2}{2}}\nabla u\right)= F_{\epsilon m} (z,u,\nabla u)+f_m\quad \text{in $Q_T$},
\\
& \text{$u=0$ on $\partial\Omega\times (0,T)$},
\\
& \text{$u(x,0)=u_{0m}(x)$ in $\Omega$},\qquad \epsilon\in (0,1),
\end{split}
\end{equation}
and the auxiliary linear problem
\begin{equation}
\label{eq:aux-1}
\begin{split}
&  u_t- (\epsilon^2+|\nabla v|^2)^{\frac{p_m-2}{2}}\sum_{i,j=1}^N
a_{ij m}(z,\nabla v) D^2_{ij}u
\\
& \qquad \qquad
= h_{\epsilon m}(z,v)u + \vec g_{\epsilon m}(z,\nabla
v)\nabla u +\tau f_m\quad \text{in $Q_T$},
\\
& \text{$u=0$ on $\partial\Omega\times (0,T)$},
\\
& \text{$u(x,0)=\tau u_{0m}(x)$ in $\Omega$},\qquad \tau\in [0,1],
\end{split}
\end{equation}
where $v\in C^{1+\alpha,\frac{1+\alpha}{2}}(\overline{Q}_T)$ is a given
function and

\[
\begin{split}
F_{\epsilon m}(z,v,\nabla v) & = h_{\epsilon m}(z,v) + (\epsilon^2+|\nabla v|^2)^{\frac{s_m(z)-2}{2}} (\vec c_m,\nabla v),
\\
a_{ijm}(z,\nabla v) & =\delta_{ij}+(p_m(z)-2)\dfrac{D_ivD_jv}{\epsilon^2+|\nabla
v|^2},
\\
h_{\epsilon m}(z,v) & = a_m(z)(\epsilon^2+|v|^{2})^{\frac{q_m(z)-2}{2}},
\\
\vec g_{\epsilon m}(z,\nabla v) & =(\epsilon^2+|\nabla
v|^2)^{\frac{p_m(z)-2}{2}}\ln(\epsilon^2+|\nabla v|^2)  \nabla p_m
+ (\epsilon^2+|\nabla v|^2)^{\frac{s_m(z)-2}{2}}\vec  c_m(z).
\end{split}
\]
If $v=u$, problem \eqref{eq:aux-1} is formally equivalent to problem \eqref{eq:main-reg-1}. For every $\epsilon>0$ and $|\nabla v|\leq M$ equation \eqref{eq:aux-1} is
uniformly parabolic because

\[
\min\{p_m-1,1\}|\vec \xi|^2 \leq \sum_{i,j=1}^Na_{ijm}\xi_i\xi_j\leq
\max\{p_m-1,1\}|\vec \xi|^2 \qquad \forall \vec\xi\in \mathbb{R}^N.
\]

\begin{proposition} \label{pro:holder-coefficients} Let conditions
\eqref{eq:p}-\eqref{eq:data} be fulfilled. For every $v\in
C^{1+\alpha,\frac{1+\alpha}{2}}(\overline Q_T)$ the coefficients of equation
\eqref{eq:aux-1} belong to $C^{\alpha,\frac{\alpha}{2}}(\overline{Q}_T)$.
\end{proposition}

\begin{proof}
Let $\mathcal{G}(p_m,s)$ be the function defined in \eqref{eq:reg-flux}. The
needed inclusions will follow is we show that $\mathcal{G}(q_m,v)$ and $\mathcal{G}(s_m,|\nabla v|)$ are H\"older-continuous as functions of $z$. For the sake
of definiteness, we consider the second case. Given two points
$z_1,z_2$, we denote $p_m(z_i)=p_{im}$, $\nabla v(z_i)=\vec \xi_i$. By the Lagrange mean value theorem

\[
\begin{split}
& \mathcal{G}(p_{2m},
\xi_2)
-
\mathcal{G}(p_{1m},\xi_1)
=
\int_0^1\dfrac{d}{d\theta}\mathcal{G}(\theta
p_{2m}+(1-\theta)p_{1m},\theta
\xi_2+(1-\theta)
\xi_1)\,d\theta
\\
&
=\frac{1}{2}(p_{2m}-p_{1m})\int_0^1\mathcal{G}(\theta
p_{2m}+(1-\theta)p_{1m},\theta
\xi_2+(1-\theta)\xi_1)\ln
(\epsilon^2+|\theta
\xi_2+(1-\theta)
\xi_1|^2)\,d\theta
\\
&
\quad
+2\sum_{i=1}^N
(\xi_{i2}-\xi_{i1})\int_0^1\dfrac{\mathcal{G}(\theta
p_{2m}+(1-\theta)p_{1m},\theta
\xi_2+(1-\theta)
\xi_1)}{\epsilon^2+|\theta
\xi_2+(1-\theta)
\xi_1|^2}\,d\theta.
\end{split}
\]
The integrands are bounded by constants depending on $\epsilon$, $ p^\pm_m$ and $|D_iv|_{0,Q_T}$, therefore

\[
\dfrac{\left|\mathcal{G}(p_m (z_2),\nabla v(z_2)) - \mathcal{G}(p_m (z_1),\nabla
v(z_1))\right|} {\rho^\alpha(z_1,z_2)}\leq
C\left([p_m ]_{\alpha,\frac{\alpha}{2}}+\sum_{i=1}^{N}[D_iv]_{\alpha,\frac{\alpha}{2}}\right).
\]
It follows that $\|\mathcal{G}(p_m ,\nabla v)\|_{\alpha,\frac{\alpha}{2},Q_T}$ is
bounded. The inclusions

\[
a_{ijm}(z,\nabla v), \,g_{\epsilon m i}(z,\nabla v), \,h_{\epsilon m}(z,v)\in
C^{\alpha,\frac{\alpha}{2}}(\overline{Q}_T)
\]
follow in the same way.
\end{proof}

Since the functions $f_{m}$ and $u_{0m}$ have finite supports in $Q_T$ and
$\Omega$, they satisfy the first-order compatibility conditions on the set
$\partial\Omega \times \{t=0\}$. By \cite[Ch.IV, Th.5.2]{LSU}
problem \eqref{eq:aux-1} has  a unique classical
solution $u\in C^{2+\alpha,1+\frac{\alpha}{2}}(\overline{Q}_T)$ which
satisfies the estimate

\begin{equation}
\label{eq:aux-2}
\|u\|_{C^{2+\alpha,1+\frac{\alpha}{2}}(\overline Q_T)}\leq C\left(\tau\|f_m\|_{C^{\alpha,\frac{\alpha}{2}}(\overline{Q}_T)}+\tau \|u_{0m}\|_{C^{2+\alpha}(\overline\Omega)}\right)
\end{equation}
with a constant $C$ depending on $\|\nabla v\|_{\alpha,\frac{\alpha}{2},Q_T}$, $N$, $\partial\Omega$, $\epsilon$,
 $p^\pm$. Problem \eqref{eq:aux-1} defines the mapping

\[
u=\Phi(v,\tau): \,(v,\tau)\mapsto u,
\]
and a fixed point of the mapping $\Phi(v,1)$ is a solution of the nonlinear problem \eqref{eq:main-reg-1}. The existence of a fixed point will follow from the Leray-Schauder principle (see, e.g., \cite[Ch.IV, Th.8.1]{LU}). Let

\[
B_R=\left\{v:\,\|v\|_{C^{1+\alpha,\frac{1+\alpha}{2}}(\overline{Q}_T)}<R\right\}
\]
be the ball of radius $R$ to be defined, and $S_R=B_R\times [0,1]$. Assume that

\begin{enumerate}
\item the equation $u=\Phi(v,0)$ only has the trivial solution,

\item $\Phi:\,C^{1+\alpha,\frac{1+\alpha}{2}}(\overline{Q}_T)\times [0,1]\mapsto C^{1+\alpha,\frac{1+\alpha}{2}}(\overline{Q}_T)$ is compact,

\item $\Phi$ is uniformly continuous in $\tau$ on $\overline{S}_R$,

\item the boundary of the ball $B_R$ does not contain fixed points of the mapping $u=\Phi(v,\tau)$.

    \end{enumerate}
Then the equation $u=\Phi(u,\tau)$ has at least one solution for every $\tau\in [0,1]$.

Let us consider $\Phi$ as the mapping

\[
\Phi:\,\overline S_R \mapsto C^{1+\alpha,\frac{1+\alpha}{2}}(\overline Q_T).%
\]

(1) Problem \eqref{eq:aux-1} with $\tau=0$ is the linear uniformly parabolic
equation with zero initial and boundary data. By the maximum principle the
unique solution of this problem is $u=0$.

(2) By virtue of \eqref{eq:aux-2}, for every $R>0$ the image of $\overline S_R$ is bounded in $C^{2+\alpha,1+\frac{\alpha}{2}}(\overline Q_T)\subset C^{1+\alpha,\frac{1+\alpha}{2}}(\overline Q_T)$ and, therefore, is compact in $C^{1+\alpha,\frac{1+\alpha}{2}}(\overline Q_T)$.

(3) Uniform continuity with respect to $\tau$. Let $u_1$, $u_2$ be the solutions of the equations $u_i=\Phi(v,\tau_i)$, $i=1,2$. Combining equations \eqref{eq:aux-1} we find that the difference $w=u_2-u_1$ is a solution of the linear problem

\[
\begin{split}
& w_t-(\epsilon^2+|\nabla v|^2)^{\frac{p_m-2}{2}}\sum_{i,j=1}^N a_{ijm}(z, \nabla v)D^2_{ij}w
\\
& \qquad \qquad
=
h_{\epsilon m}(z,v)w+\vec g_{\epsilon m}(z,\nabla v)\cdot \nabla w+(\tau_2-\tau_1)f_m,
\\
& \text{$w(x,0)=(\tau_2-\tau_1)u_{0m}$ in $\Omega$},\quad \text{$w=0$ on $\partial\Omega\times (0,T)$}.
\end{split}
\]
By virtue of \eqref{eq:aux-2} $\|w\|_{C^{2+\alpha,1+\frac{\alpha}{2}}}(\overline{Q}_T)\to 0$ as $|\tau_2-\tau_1|\to 0$.

(4) To check that the boundary of the set $S_R$ does not contain fixed points of the mapping $\Phi$ suffices to prove that all possible fixed points belong to a ball $B_{R'}$ of a smaller radius $R'<R$. The proof amounts to deriving the a priori estimate

\begin{equation}
\label{eq:Holder-main}
\|u\|_{C^{1+\alpha,\frac{1+\alpha}{2}}(\overline{Q}_T)}\leq R'
\end{equation}
for all possible fixed points of $\Phi(\tau,v)$ and choosing $R=R'+1$. Let
$u\in C^{1+\alpha,\frac{1+\alpha}{2}}(\overline{Q}_T)$ be a fixed point of the
mapping $\Phi(v,\tau)$ with some $\tau\in [0,1]$. On the one hand, $u$ is a
classical solution of the linear equation \eqref{eq:aux-1} with the smooth
data $\tau f_m\in C_0^{\infty}(Q_T)$ and $\tau u_{0m}\in C_0^\infty(\Omega)$,
$\tau\in [0,1]$. On the other hand, $u$ is a strong solution of the nonlinear
problem \eqref{eq:main-reg-1}.

The variable nonlinearity of equation \eqref{eq:main-reg-1} prevents one from deriving \eqref{eq:Holder-main} from the classical parabolic theory. To obtain \eqref{eq:Holder-main}, we make use of the recent results on the gradient regularity of weak solutions to nonlinear equations of the type \eqref{eq:main-reg}. It is proven in \cite{Yao-2015} that for the solutions of equation
\eqref{eq:main} with a given source term in divergence form
$\operatorname{div}(|\mathbf{f}|^{p(z)-2}\mathbf{f})$ one has $\nabla u\in
C^{\gamma,\frac{\gamma}{2}}_{loc}(Q_T)$, provided that the components of $\mathbf{f}$ are H\"older continuous. This local result
allows one to estimate the H\"older norm of $\nabla u$ in every domain of the
form $\overline{\Omega}\times (\delta,T)$ with smooth $\partial\Omega$ and any $\delta>0$ - \cite{NA-Sh-2018}. This is done by means of local rectification of the boundary $\partial\Omega$ and
application of the local result of \cite{Yao-2015} to the odd continuation of
the solution across the lateral boundary to the exterior of the problem domain, see \cite[Lemma 1.3, (2)]{NA-Sh-2018}.

In \cite{Ding-Zhang-Zhou-2020}, the results of \cite{Yao-2015, NA-Sh-2018}
were extended to a class of equations with nonlinear sources of general form
that includes \eqref{eq:source} as a partial case. Let us consider the problem

\begin{equation}
\label{eq:Zhang-2020}
\begin{split}
& u_t-\operatorname{div} \mathcal{A}(z,\nabla u)\nabla u=\mathcal{B}(z,u,\nabla u) \quad \text{in $Q_T$},
\\
& \text{$u=0$ on $\partial\Omega\times (0,T)$},\qquad \text{$u(x,0)=u_0(x)$ in $\Omega$},
\end{split}
\end{equation}
where

\begin{equation}
\label{eq:Zhang-cond}
\begin{split}
& \text{$\mathcal{A}(z,\nabla u)\equiv |\nabla u|^{p_m(z) -2}\nabla u$},
\\
& \text{$\mathcal{B}$ satisfies the growth condition}
\\
& \qquad |\mathcal{B}(z,u,\xi)|\leq
b_1|\xi|^{p_m(z) -1}+b_2|u|^{q_m(z)-1}+b_3, \quad b_1,b_2,b_3=const\geq 0,
\\
& \text{the exponents $p_m(z)$, $q_m(z)$ satisfy conditions \eqref{eq:p}, \eqref{eq:q} and \eqref{eq:data-reg}}.
\end{split}
\end{equation}
A function $u\in C^0([0,T];L^2(\Omega))\cap
\mathbf{W}_{p_m(\cdot)}(Q_T)\cap L^{q_m(\cdot)}(Q_T)$ is called a weak solution of problem \eqref{eq:Zhang-2020} if

\begin{equation}
\label{eq:def-weak}
\int_{Q_T}\left(-u\phi_t+\mathcal{A}(z,\nabla u)\cdot \nabla \phi -\mathcal{B}(z,u,\nabla u)\phi\right)\,dz=\int_{\Omega} \phi(x,0)u_0\,dx
\end{equation}
for every $\phi\in L^{q_m(\cdot)}(Q_T)\cap \mathbf{W}_{p_m(\cdot)}(Q_T)$ with
$\phi_t\in L^2(Q_T)$ and $\phi(x,T)=0$.

\begin{proposition}[\cite{Ding-Zhang-Zhou-2020}, Th.1]
\label{pro:max}
Let conditions \eqref{eq:Zhang-cond} be fulfilled, $p_m\in C_{{\rm log}}(\overline Q_T)$ and $\partial\Omega\in C^1$. There exist positive constants $\gamma$, $C$, depending only on the constants in conditions \eqref{eq:p} and \eqref{eq:Zhang-cond} such that for every weak solution $u$ of problem \eqref{eq:Zhang-2020}

\begin{equation}
\label{eq:max}
\|u(\cdot,t)\|_{\infty,\Omega}\leq \|u_0\|_{\infty,\Omega} + C\left(1+ \left(\int_{Q_T}|u|^{\lambda(z)}\,dz\right)^{\gamma}\right)
\end{equation}
with $\lambda(z)=\max\{2,p_m(z) \}$.
\end{proposition}

\begin{proposition}[\cite{Ding-Zhang-Zhou-2020}, Th.2]
\label{pro:local/global}
Let $u$ be a bounded weak solution of problem \eqref{eq:Zhang-2020}. Assume that conditions \eqref{eq:Zhang-cond} are fulfilled, $p_m \in C^{\alpha,\frac{\alpha}{2}}(\overline Q_T)$, and $\partial\Omega\in C^{1+\beta}$, $\alpha,\beta\in (0,1)$.  Then for any $\delta>0$ there exists $\gamma\in (0,1)$ such that $\nabla u\in C^{\gamma,\frac{\gamma}{2}}([\delta,T]\times \overline\Omega)$.
\end{proposition}
A revision of the proof of Proposition \ref{pro:local/global} shows that the assertion holds true for $\mathcal{A}(z,\nabla u)$ substituted by $\mathcal{F}_{\epsilon m}(z,\nabla u)$ with $\epsilon>0$.

\begin{lemma}
\label{le:first-est} Under conditions \eqref{eq:p}-\eqref{eq:data}  and
\eqref{eq:data-reg} every strong solution of problem \eqref{eq:main-reg-1}
with the data $\tau f_m$, $\tau u_{0m}$ and $\textbf{data}_m$ satisfies the
estimate

\begin{equation}
\label{eq:est-first}
\begin{split}
\sup_{(0,T)}\|u\|^2_{2,\Omega}+\int_{Q_T}|\nabla u|^{p_m(z)}\,dz & \leq
C\left(1+\tau^2\|f_m\|_{2,\Omega}^{2}+\tau^2\|u_{0m}\|_{2,\Omega}^{2}\right)
\\
&
\leq
C'\left(1+\|f\|_{2,\Omega}^{2}+\|u_{0}\|_{2,\Omega}^{2}\right)
\end{split}
\end{equation}
with a constant $C'$ that depends on $N$, $p^-$, $q^+$, $s^+$,
$T$, $|\Omega|$, but independent of $\epsilon$, $\tau$, and $m$.
\end{lemma}

\begin{proof}
Take $t,t+h\in [0,T]$, $h>0$. Let $\chi_{[t,h+h]}$ be the characteristic function of the interval $[t,t+h]\subset (0,T)$. Choose in \eqref{eq:def-1} $\phi=u\chi_{[t,t+h]}$, divide the resulting equality by $h$, and send $h\to 0^+$. By the Lebesgue differentiation theorem, for a.e. $t\in (0,T)$

\begin{equation}
\label{eq:1st-ODI}
\begin{split}
\frac{1}{2} \dfrac{d}{dt}\|u(t)\|_{2,\Omega}^{2} & +\int_{\Omega}(\epsilon^2+|\nabla u|^2)^{\frac{p_m-2}{2}}|\nabla u|^2\,dx
\\
&
=\int_{\Omega}F_{\epsilon m}(z,u,\nabla u)\,dx +\tau\int_{\Omega}f_mu\,dx
\\
& \leq C\bigg(\tau^2\|f_m\|_{2,\Omega}^2+\|u\|_{2,\Omega}^2 +\int_{\Omega}(\epsilon^2+|u|^2)^{\frac{q_m-2}{2}}|u|^{2}\,dx \\
& \qquad \qquad + \int_{\Omega}(\epsilon^2+|\nabla u|^2)^{\frac{s_m-2}{2}}|\nabla u||u|\,dx\bigg)
\\
& \equiv C\left(\tau^2\|f_m\|_{2,\Omega}^2+\|u\|_{2,\Omega}^2\right)+C\Psi(t)
\end{split}
\end{equation}
with a constant $C=C(a^+,c_i^+)$. Inequality \eqref{eq:1st-ODI} can be written as

\[
\dfrac{d}{dt}\left({\rm e}^{-Ct}\|u(t)\|_{2,\Omega}^2\right)+ 2{\rm e}^{-Ct}\int_{\Omega} (\epsilon^2+|\nabla u|^2)^{\frac{p_m-2}{2}}|\nabla u|^2\,dx\leq 2C\tau^2{\rm e}^{-Ct}\|f_m\|_{2,\Omega}^2+2C{\rm e}^{-Ct}\Psi(t).
\]
Integration in $t$ leads to the inequality

\begin{equation}
\label{eq:1st-est}
\begin{split}
\sup_{(0,T)}\|u(t)\|^2_{2,\Omega} & + \int_{Q_T}(\epsilon^2+|\nabla u|^2)^{\frac{p_m-2}{2}}|\nabla u|^2\,dz
\\
& \leq C'{\rm e}^{CT}\left(\|u_{0m}\|_{2,\Omega}^{2}+ \tau^2 \|f_m\|^2_{2,Q_T}
+\int_0^T\Psi(t)\,dt\right)
\\
& \leq C''{\rm e}^{CT}\left(1+\|u_{0}\|_{2,\Omega}^{2}+\|f\|^2_{2,Q_T}+ \int_0^T\Psi(t)\,dt\right)
\end{split}
\end{equation}
with an independent of $u$ constant $C''$. At every point where $p_m(z) \geq 2$ we have

\[
|\nabla u|^{p_m}\leq (\epsilon^2 +|\nabla u|^2)^{\frac{p_m-2}{2}}|\nabla u|^{2}.
\]
If $p_m(z) <2$, then

\begin{equation}
\label{eq:elem-1}
\begin{split}
|\nabla u|^{p_m} & \leq \begin{cases}
2^{\frac{2-p_m}{2}}(\epsilon^2 +|\nabla u|^2)^{\frac{p_m-2}{2}}|\nabla u|^2 & \text{if $\epsilon^2<|\nabla u|^2$},
\\
\epsilon^p_m & \text{if $\epsilon^2\geq |\nabla u|^2$}
\end{cases}
\\
&
\leq 1+ 2(\epsilon^2 +|\nabla u|^2)^{\frac{p_m-2}{2}}|\nabla u|^2.
\end{split}
\end{equation}
This observation allows one to rewrite \eqref{eq:1st-est} as
\begin{equation}
\label{eq:1st-est-prim}
\sup_{(0,T)}\|u(t)\|^2_{2,\Omega} + C\int_{Q_T}|\nabla u|^{p_m}\,dz\leq C'{\rm e}^{CT}\left(1+\|u_{0}\|_{2,\Omega}^{2} + \|f\|_{2,Q_T}^2+\int_0^T\Psi(t)\,dt\right).
\end{equation}
To estimate the first term of $\Psi(t)$ we use the assumption $q_m^+\leq p_m^--\mu$, apply the Poincar\'e inequality with the constant exponent $p_m^--\mu$, and then make use of the Young inequality: for any $\delta>0$

\begin{equation}
\label{eq:est-1}
\begin{split}
\int_{\Omega}(\epsilon^2+|u|^2)^{\frac{q_m-2}{2}}u^2\,dx & \leq C\left(1+\int_{\Omega}|u|^{ q_m^+}\,dx\right)\leq  C\left(1+\int_\Omega |u|^{p_m^--\mu}\,dx\right)
\\
&
\leq C'\left(1+\int_{\Omega}|\nabla u|^{p_m^--\mu}\,dx \right)
\leq C''+\delta \int_{\Omega}|\nabla u|^{p_m}\,dx.
\end{split}
\end{equation}
For a sufficiently small $\delta$, after integration in $t$ this term is absorbed in the left-hand side of \eqref{eq:1st-est-prim}. The second term of $\Psi(t)$ is estimated likewise: by Young's inequality

\[
\begin{split}
\int_\Omega |u|(\epsilon^2+|\nabla u|^{2})^{\frac{s_m-2}{2}}|\nabla u|\,dx
& \leq \int_{\Omega}|u| (\epsilon^2+|\nabla u|^{2})^{\frac{s_m-1}{2}}\,dx
\\
&
\leq C_\delta\int_\Omega |u|^{p_m^--\mu}\,dx + \delta \int_\Omega (\epsilon^2+|\nabla u|^2)^{\frac{s_m-1}{2}(p_m^--\mu)'}\,dx.
\end{split}
\]
The first integral is already estimated in \eqref{eq:est-1}. To estimate the second one we notice that the assumption $s_m(z)\leq p_m^--\mu$ yields

\[
(s_m-1)(p_m^--\mu)'=(s_m-1)\frac{p_m^--\mu}{p_m^--\mu-1}\leq p_m^--\mu
\]
and once again apply Young's inequality:

\[
\begin{split}
|\nabla u|^{(s_m-1)(p_m^--\mu)'} & \leq |\nabla u|^{p_m^--\mu} + 1\leq \delta |\nabla u|^{p_m}+C_\delta
,\qquad \text{with any $\delta>0$}.
\end{split}
\]
It follows that

\[
C\int_{0}^T\Psi(t)\,dt\leq C+\delta \int_{Q_T}|\nabla u|^{p_m}\,dz
\]
with any $\delta>0$ and a constant $C$ depending on $T$, $|\Omega|$, $p^\pm$, $\mu$, but independent of $\epsilon$ and $m$. Choosing $\delta$ sufficiently small and moving  the terms containing $|\nabla u|^{p_m}$ to the left-hand side of \eqref{eq:1st-est-prim}  we arrive at \eqref{eq:est-first}.
\end{proof}

\begin{proposition}[\cite{SSS-2022}, Lemma 4.4]
\label{pro:extra-int}
Assume that $\partial \Omega\in Lip$, $p\in C^{0}(\overline Q_T)$, $u\in L^\infty(0,T;L^2(\Omega))\cap \mathbf{W}_{p(\cdot)}(Q_T)$. If

\[
\operatorname{ess}\sup_{(0,T)}\|u\|_{2,\Omega}^2+\int_{Q_T}|\nabla u|^{p(z) }\,dz\leq M,
\]
then for every $\lambda \in \left[\left.0,\frac{1}{N}\right)\right.$

\[
\|u\|_{p(\cdot)+\lambda,Q_T}\leq C
\]
with a constant $C$ depending on $M$, $\lambda$, $N$, $|\Omega|$, $p^\pm$ and the modulus of continuity of $p$ in $Q_T$.
\end{proposition}

\begin{proposition}
\label{pro:max-reg} Under the assumptions of Lemma \ref{le:first-est}
every weak solution of problem \eqref{eq:main-reg-1} satisfies the estimate

\[
\operatorname{ess}\sup_{\Omega}|u(\cdot,t)|\leq \tau \operatorname{ess}\sup_{\Omega}|u_{0m}| + C
\]
with a constant $C$ which depends on $\tau\|f_m\|_{2,Q_T}$ and on the same quantities as the constant $C'$ in Lemma \ref{le:first-est},
but does not depend on $\epsilon$ and $m$.
\end{proposition}

\begin{proof}
Notice that

\[
\int_{Q_T}|u|^{\max\{2,p_m(z) \}}\,dz\leq T\operatorname{ess}\sup_{(0,T)}\|u(t)\|_{2,\Omega}^2+\int_{Q_T}|u|^{p_m(z) }\,dz.
\]
By Lemma \ref{le:first-est} and Proposition \ref{pro:extra-int} both integrals are bounded, and the conclusion follows from Proposition \ref{pro:max}.
\end{proof}

We are now in position to derive \eqref{eq:Holder-main} and estimate all
possible fixed points of $\Phi(\tau,v)$ in the norm of
$C^{1+\beta,\frac{1+\beta}{2}}(\overline Q_T)$ with some $\beta\in (0,1)$. Let
$u$ be a fixed point of the mapping $\Phi(\tau,v)$, that is, a strong solution
of problem \eqref{eq:main-reg-1} with the data $\tau f_m$, $\tau u_{0m}$,
$\tau\in [0,1]$, and $\textbf{data}_m$. Since every strong solution is a weak
solution in the sense of \eqref{eq:def-weak}, by Proposition \ref{pro:max-reg}
$u$ is bounded in $\overline{Q}_T$. Fix some $\delta>0$. By Proposition
\ref{pro:local/global} there is $\gamma\in (0,1)$ such that $\nabla u\in
C^{\gamma,\gamma/2}([\delta,T]\times \overline\Omega)$. Set

\[
w(x,t)=\begin{cases}
u(x,t) & \text{if $t\geq 0$},
\\
\tau u_{0m} & \text{if $-T\leq t\leq 0$}.
\end{cases}
\]
The function $w$ is a weak solution of the equation

\[
w_t- \operatorname{div}\mathcal{F}_{\epsilon m}(z,\nabla w)= \mathcal{B}_{\epsilon m}(z,w,\nabla w)\quad \text{in $\Omega\times (-T,T)$}
\]
with

\[
\mathcal{B}_{\epsilon m}(z,w,\nabla w)=\begin{cases}
 h_{\epsilon m}(z,u)u + \vec g_{\epsilon m}(z,\nabla
u)\nabla u +\tau f_m & \text{if $t>0$},
\\
- \operatorname{div}\mathcal{F}_{\epsilon m}((x,0), \tau \nabla u_{0m}) & \text{if $t\leq 0$}.
\end{cases}
\]
Since $u_{0m}\in C_0^\infty(\Omega)$, we know that $\mathcal{B}_{\epsilon m}(z,w,\nabla w)$ satisfies conditions \eqref{eq:Zhang-cond}. Applying Proposition \ref{pro:local/global} to $w$ in the cylinder $\Omega\times [-T,T]$, for the function $u=w|_{\overline Q_T}$ we obtain the following global estimate.

\begin{lemma}
\label{le:global-grad} If conditions \eqref{eq:p}-\eqref{eq:data} are
fulfilled and $\partial\Omega\in C^{2+\alpha}$, then there exists $\gamma\in
(0,1)$ such that for every  strong solution of problem \eqref{eq:main-reg-1}
with the data $\tau f_m$, $\tau u_{0m}$, $\tau\in [0,1]$ and $\textbf{data}_m$

\[
\nabla u\in C^{\gamma,\gamma/2}(\overline{Q}_T),\qquad \|\nabla u\|_{C^{\gamma,\gamma/2}(\overline{Q}_T)}\leq C'
\]
with a constant $C'$ depending only on the data.
\end{lemma}

Now we revert to equation \eqref{eq:aux-1} with $v=u$ and consider it as a linear equation with the coefficients and the right-hand side in $C^{\gamma,\gamma/2}(\overline{Q}_T)$. The solution of problem \eqref{eq:aux-1} belongs to $C^{2+\gamma,1+\frac{\gamma}{2}}(\overline Q_T)$, by virtue of \eqref{eq:aux-2} $\|u\|_{2+\gamma,\frac{1+\gamma}{2},Q_T}\leq C''$ with a constant $C''$ depending only on the data. To conclude the proof, it suffices to take $R=C''+1$. We summarize these arguments in the following assertion.

\begin{lemma}
\label{le:existence-reg}
If conditions \eqref{eq:p}-\eqref{eq:data} are fulfilled and $\partial\Omega\in C^{2+\alpha}$, then problem \eqref{eq:main-reg-1} with $\textbf{data}_m$ has a unique classical solution $u\in C^{2+\gamma,1+\frac{\gamma}{2}}(\overline Q_T)$ with some $\gamma\in (0,1)$. Moreover, for every $t\in [0,T]$ $u(\cdot,t)\in C^3(\Omega)$.
\end{lemma}

\begin{proof}
The existence of a classical solution $u$ is already proven. Fix a direction $k=1,\ldots,N$ and consider the function $w\equiv D_ku$. Differentiation of the equation for $u$ with respect to $x_k$ shows that $w=D_ku$ satisfies the equation

\begin{equation}
\label{eq:w}
\begin{split}
w_t & -(\epsilon^2+|\nabla u|^2)^{\frac{p_m-2}{2}}\sum_{i,j=1}^N a_{ijm}(z, \nabla u) D^2_{ij}w = D_kf_m
\\
& + D_k\left((\epsilon^2+|\nabla u|^2)^{\frac{p_{m}-2}{2}}\ln(\epsilon^2+|\nabla u|^2)(\nabla u,\nabla p_m)\right)
\\
&
+ \sum_{i,j=1}^N D_k\left((\epsilon^2+|\nabla u|^2)^{\frac{p_{m}-2}{2}}a_{ijm}\right)D^2_{ij}u
\\
& + D_k\left(a_m(\epsilon^2+|u|^2)^{\frac{q_{m}(z)-2}{2}}u\right) + D_k\left((\epsilon^2+|\nabla u|^2)^{\frac{s_{m}(z)-2}{2}}(\vec c_m,\nabla u)\right).
\end{split}
\end{equation}
This is a linear uniformly parabolic equation with the coefficients and the right-hand side in $C^{\gamma,\gamma/2}(\overline{Q}_T)$. By \cite[Ch.III, Th.12.1]{LSU} $w\in C^{2+\gamma,1+\frac{\gamma}{2}}(\Omega\times (0,T])$. Hence, $u(\cdot,t)\in C^3(\Omega)$ for every $t\in [0,T]$.
\end{proof}

\section{Inequalities for the regularized flux}
\label{sec:flux}

\subsection{Formulas of integration by parts}
Let the conditions of Lemma \ref{le:existence-reg} be fulfilled and $u$ be the
classical solution of problem \eqref{eq:main-reg-1}. For every $t_0\in [0,T]$,
$u(\cdot,t_0)\in C^{2+\gamma}(\overline \Omega)\cap C^3(\Omega)$ and
\[
\mathcal{F}_{\epsilon}(z,\nabla u)\in C^1(\overline\Omega)\cap C^2(\Omega),\quad \text{for $t=t_0$}.
\]
In the local coordinates $\{y_i\}_{i=1}^{N-1}$ the surface $\partial\Omega$ is represented by the equation $y_N=\phi(y')$. Let us denote by $\mathcal{B}(\xi;\eta)$ the second quadratic form of the surface $\partial \Omega$: for every two tangent vectors $\vec \xi$, $\vec \eta$ to $\partial\Omega$

\begin{equation}
\label{eq:fund-form}
\mathcal{B}(\xi;\eta)=
\sum_{i,j=1}^{N-1}\dfrac{\partial^2\phi}{\partial y_i\partial y_j}(0)\xi_i\eta_j,\qquad
\operatorname{trace}\mathcal{B}=\sum_{i=1}^{N-1}D^2_{y_iy_i}\phi(0).
\end{equation}
Let $\vec \nu$ be the exterior normal to $\partial\Omega$ and $\vec \tau$ belong to the tangent plane to $\partial\Omega$ at the same point. Since $u=0$ on $\partial\Omega$, then $\vec \nu=\dfrac{\nabla u}{|\nabla u|}$ and

\[
\mathcal{F}_\epsilon(z,\nabla u)\cdot \vec \nu|_{\partial \Omega}=(\epsilon^2+|\nabla u|^2)^{\frac{p-2}{2}}|\nabla u|,\qquad \mathcal{F}_\epsilon(z,\nabla u)\cdot \vec \tau|_{\partial \Omega}=0.
\]
Then by \cite[Th.3.1.1.1]{Grisvard}

\begin{equation}
\label{eq:double-final-e}
\begin{split}
\int_{\Omega} (\operatorname{div}\mathcal{F}_\epsilon(z,\nabla u))^2\,dx  = & - \int_{\partial\Omega}(\mathcal{F}_\epsilon(z,\nabla u),\vec \nu)^2
\operatorname{trace}\mathcal{B}\,dS
\\
& +\sum_{i,j=1}^N\int_{\Omega}D_{i}\left(\mathcal{F}^{(j)}_\epsilon(z,\nabla u)\right) D_{j}\left(\mathcal{F}^{(i)}_\epsilon(z,\nabla u)\right)\,dx,
\end{split}
\end{equation}
where $t$ is fixed and $\mathcal{F}_\epsilon^{(k)}$ denotes the $k$th component of $\mathcal{F}_\epsilon$.

Since for the smooth convex domain,
$\operatorname{trace}\mathcal{B}=\sum_{i=1}^{N-1}D^2_{y_iy_i}\phi(0)\leq 0$,
then
\begin{equation}
\label{eq:double-convex}
\begin{split}
\int_{\Omega} (\operatorname{div}\mathcal{F}_\epsilon(z,\nabla u_k))^2\,dx
\geq  &
\sum_{i,j=1}^N\int_{\Omega}D_{i}\left(\mathcal{F}^{(j)}_\epsilon(z,\nabla
u_k)\right) D_{j}\left(\mathcal{F}^{(i)}_\epsilon(z,\nabla u_k)\right)\,dx.
\end{split}
\end{equation}

\subsection{A differential inequality}
The integrands of the integrals over $\Omega$ on the right-hand side of \eqref{eq:double-final-e} can be transformed in the following manner. Given $\nabla u$, we denote
\[
\vec{\eta}=\dfrac{\nabla u}{(\epsilon^2+|\nabla u|^2)^{\frac{1}{2}}},\qquad |\vec \eta|<1.
\]
A straightforward computation shows that\[
\begin{split}
D_{i} & \left(\mathcal{F}^{(j)}_\epsilon(z,\nabla u)\right)
= (\epsilon^2+|\nabla u|^2)^{\frac{p-2}{2}}D^2_{ij}u
+ (p-2)(\epsilon^2+|\nabla u|^2)^{\frac{p-2}{2}-1}D_ju\sum_{k=1}^ND_k u D^2_{ki}u
\\
&
\qquad +\frac{1}{2}(\epsilon^2+|\nabla u|^2)^{\frac{p-2}{2}}\ln (\epsilon^2+|\nabla u|^2) D_juD_ip
\\
& = (\epsilon^2+|\nabla u|^2)^{\frac{p-2}{2}}
\\
& \quad \times
\left(D^2_{ij}u + (p-2)\frac{D_ju}{(\epsilon^2+|\nabla u|^2)^{\frac{1}{2}}}\sum_{k=1}^N D^2_{ki}u\frac{D_ku}{(\epsilon^2+|\nabla u|^2)^{\frac{1}{2}}}+ \frac{1}{2}\ln (\epsilon^2+ |\nabla u|^2) D_juD_ip\right)
\\
& =
(\epsilon^2+|\nabla u|^2)^{\frac{p-2}{2}}\left(D^2_{ij}u + (p-2)\eta_j\sum_{k=1}^N D^2_{ki}u\eta_k+ \frac{1}{2}\ln (\epsilon^2+ |\nabla u|^2) D_juD_ip\right).
\end{split}
\]
For every $i,j=\overline{1,N}$
\[
\begin{split}
D_{i} & \left(\mathcal{F}^{(j)}_\epsilon(z,\nabla u)\right)
D_j\left(\mathcal{F}^{(i)}_\epsilon(z,\nabla u)\right)
\\
& = (\epsilon^2+|\nabla u|^2)^{p-2} \left(D^2_{ij}u + (p-2)\eta_i\sum_{k=1}^N D^2_{kj}u\eta_k+ \frac{1}{2}(\epsilon^2+|\nabla u|^2)^{\frac{1}{2}}\ln (\epsilon^2+ |\nabla u|^2) \eta_iD_jp\right)
\\
& \quad \times \left(D^2_{ij}u + (p-2)\eta_j\sum_{k=1}^N D^2_{ki}u\eta_k+ \frac{1}{2}(\epsilon^2+|\nabla u|^2)^{\frac{1}{2}}\ln (\epsilon^2+ |\nabla u|^2) \eta_jD_ip\right)
\\
& \equiv (\epsilon^2+|\nabla u|^2)^{p-2}\mathcal{K}_{ij}.
\end{split}
\]
Denote by $\mathcal{H}(u)$ the Hessian matrix of $u$: the symmetric $N\times N$ matrix with the entries $\mathcal{H}_{ij}(u)=D^2_{ij}u$, $i,j=\overline{1,N}$. With this notation, the expression for $\mathcal{K}_{ij}$ becomes
\[
\begin{split}
\mathcal{K}_{ij} & = \mathcal{H}^2_{ij}(u) + (p-2)\mathcal{H}_{ij}(u)\left[\sum_{k=1}^N(\eta_j\mathcal{H}_{ki}(u)+\eta_i \mathcal{H}_{kj}(u))\eta_k\right]
\\
& + (p-2)^2\eta_i\eta_j \sum_{k,l=1}^N
\mathcal{H}_{li}(u)\mathcal{H}_{kj}(u)\eta_l\eta_k
\\
&
+ \frac{1}{2}\mathcal{H}_{ij}(u)\ln(\epsilon^2+|\nabla u|^2)\left(D_iuD_jp+D_juD_ip\right)
\\
& + \frac{1}{2}(p-2)\eta_i\eta_j(\epsilon^2+|\nabla u|^2)^{\frac{1}{2}}\sum_{k=1}^N\left(\mathcal{H}_{ki}(u)  + \mathcal{H}_{ki}(u)\right)\eta_k \ln (\epsilon^2+|\nabla u|^2)(D_ip+D_jp)
\\
& + \frac{1}{4}(\epsilon^2+|\nabla u|^2)\ln^2(\epsilon^2+|\nabla u|^2)\eta_i\eta_jD_ipD_jp
\\
& \equiv \sum_{s=1}^{6}\mathcal{J}^{(s)}_{ij}.
\end{split}
\]
Summing up we obtain the representation
\[
\begin{split}
\sum_{i,j=1}^N \mathcal{K}_{ij} & =\sum_{i,j=1}^N\mathcal{H}^2_{ij}(u)
+ (p-2)\sum_{i,j=1}^N\mathcal{H}_{ij}(u)\left[\sum_{k=1}^N(\eta_j\mathcal{H}_{ki}(u)+\eta_i \mathcal{H}_{kj}(u))\eta_k\right]
\\
& \qquad + (p-2)^2\sum_{i,j=1}^N\eta_i\eta_j \sum_{k,l=1}^N
\mathcal{H}_{li}(u)\mathcal{H}_{kj}(u)\eta_l\eta_k +\sum_{i,j=1}^N\sum_{s=4}^6\mathcal{J}_{ij}^{(s)}
\\
& = \sum_{i,j=1}^N\mathcal{H}^2_{ij}(u) + 2(p-2)\sum_{i=1}^N\left(\sum_{j=1}^N\mathcal{H}_{ij}(u)\eta_j\right)\left(\sum_{k=1}^N \mathcal{H}_{ik}(u)\eta_k\right)
\\
& \qquad +(p-2)^2 \sum_{i,j=1}^N\left(\sum_{l=1}^N
\mathcal{H}_{li}(u)\eta_l\eta_i\right) \left(\sum_{k=1}^{N}\mathcal{H}_{kj}(u)\eta_j\eta_k\right) +\sum_{i,j=1}^N\sum_{s=4}^6\mathcal{J}_{ij}^{(s)}
\\
& = \sum_{i,j=1}^N\mathcal{H}^2_{ij}(u) + 2(p-2)\sum_{i=1}^N\left(\sum_{j=1}^N\mathcal{H}_{ij}(u)\eta_j\right)\left(\sum_{k=1}^N \mathcal{H}_{ik}(u)\eta_k\right)
\\
& \qquad +(p-2)^2 \left(\sum_{i,l=1}^N
\mathcal{H}_{li}(u)\eta_l\eta_i\right) \left(\sum_{j,k=1}^{N}\mathcal{H}_{kj}(u)\eta_j\eta_k\right) +\sum_{i,j=1}^N\sum_{s=4}^6\mathcal{J}_{ij}^{(s)}
\\
& = \operatorname{trace}\mathcal{H}^2(u) + 2(p-2)|(\mathcal{H}(u),\vec \eta)|^2+(p-2)^2 \left((\mathcal{H}(u),\vec \eta)\cdot \vec\eta\right)^2 +\sum_{i,j=1}^N\sum_{s=4}^6\mathcal{J}_{ij}^{(s)}.
\end{split}
\]
Accept the notation
\[
\mathcal{M}(\vec \eta)\equiv \operatorname{trace}\mathcal{H}^2(u) + 2(p-2)|(\mathcal{H}(u),\vec \eta)|^2+(p-2)^2 \left((\mathcal{H}(u),\vec \eta)\cdot \vec \eta\right)^2.
\]
We want to show that there is a constant $C>0$ such that
\begin{equation}
\label{eq:est-from-below}
\mathcal{M}(\vec \eta)\geq C\operatorname{trace}\mathcal{H}^2(u)\equiv C|u_{xx}|^2\qquad \forall\,\vec \eta \in \mathbb{R}^N,\;\;|\vec \eta|=\lambda
\end{equation}
with $\lambda< 1$ but sufficiently close to $1$. For $|\vec \eta|=1$ inequality \eqref{eq:est-from-below} coincides with inequality (3.5) in \cite{Cianchi-Maz'ya-2018}. Its proof can be easily adapted to our case where $|\vec \eta|<1$. Fix a point $x_0$, diagonalize the matrix $\mathcal{H}(u(x_0))$ by means of rotation, and denote by $\vec\zeta$ the vector $\vec\eta$ in the new basis:
\[
(\mathcal{H}(u(x_0))\cdot \vec\eta )\cdot (\mathcal{H}(u(x_0))\cdot \vec\eta )=\sum_{i=1}^Nd_i^2\zeta_i^2,\quad \mathcal{H}(u(x_0))\cdot\vec \eta)\cdot \vec \eta=\sum_{i=1}^N d_i\zeta_i^2,\quad \operatorname{trace}\mathcal{H}^2(u)=\sum_{i=1}^N d_i^2,
\]
where $d_i$ are the diagonal elements of the transformed matrix. Set $|\vec\zeta|=\lambda<1$ and denote $\vec \xi=\lambda^{-1}\vec\zeta$, $|\vec \xi|=1$. Then
\[
\begin{split}
\mathcal{M}(\vec\eta) &  = \sum_{i=1}^Nd_i^2+2(p-2)\sum_{i=1}^N d^2_i\zeta_i^2 + (p-2)^2 \left(\sum_{i=1}^N d_i\zeta_i^2\right)^2,
\\
& = \sum_{i=1}^Nd_i^2+2\lambda^2(p-2)\sum_{i=1}^N d_i\xi_i^2 + \lambda^4(p-2)^2 \left(\sum_{i=1}^N d_i^2\xi_i^2\right)^2
\equiv \Phi(\vec \xi,\lambda).
\end{split}
\]
The function $\Phi(\vec \xi,\lambda)$ is continuous with respect to $\lambda$. By inequality (3.5) in \cite{Cianchi-Maz'ya-2018} we know that there exists a constant $C$ such that $\Phi(\vec \xi,1)\geq C\operatorname{trace}\mathcal{H}^2(u(x_0))$.  Hence, there is $\epsilon_0\in (0,1)$ such that for $\epsilon\in (0,\epsilon_0)$ the parameter $\lambda$ is so close to one that
$\Phi(\vec \xi,\lambda)\geq \dfrac{C}{2}\operatorname{trace}\mathcal{H}^2(u(x_0))$. Since $x_0\in \Omega$ is arbitrary, it follows that there exists a constant $C>0$ such that for all $x\in \Omega$
\begin{equation}
\label{eq:reduction}
\begin{split}
\sum_{i,j=1}^N & D_{i} \left(\mathcal{F}^{(j)}_\epsilon(z,\nabla u)\right)
D_j\left(\mathcal{F}^{(i)}_\epsilon(z,\nabla u)\right)
\\
&
\geq C(\epsilon^2+|\nabla u|^2)^{p-2}|u_{xx}|^2 +(\epsilon^2+|\nabla u|^2)^{p-2}\sum_{i,j=1}^N\sum_{s=4}^6\mathcal{J}_{ij}^{(s)},
\end{split}
\end{equation}
where
\[
\begin{split}
& \mathcal{J}^{(4)}_{ij}=\frac{1}{2}\mathcal{H}_{ij}(u)\ln(\epsilon^2+|\nabla u|^2)\left(D_iuD_jp+D_juD_ip\right),
\\
& \mathcal{J}^{(5)}_{ij}= (p-2)\eta_i\eta_j(\epsilon^2+|\nabla u|^2)^{\frac{1}{2}}\sum_{k=1}^N  \mathcal{H}_{ki}(u) \eta_k \ln (\epsilon^2+|\nabla u|^2)(D_ip+D_jp)
\\
& \mathcal{J}^{(6)}_{ij}=\frac{1}{4}\ln^2(\epsilon^2+|\nabla u|^2)(\epsilon^2+|\nabla u|^2)\eta_i\eta_jD_ipD_jp.
\end{split}
\]

\subsection{Integral inequalities}
Plugging \eqref{eq:reduction} into \eqref{eq:double-final-e} we arrive at the following assertion.

\begin{lemma}
\label{le:principal-e}
Let $\partial\Omega\in C^2$, $u\in C^3(\Omega)\cap C^{2}(\overline{\Omega})$, $p\in C^2(\overline{\Omega})$. Assume that $p\in C^2(\overline{\Omega})\cap C^{0,1}(\overline \Omega)$ with the Lipschitz
constant $L$. There exists a positive constant $C=C(N,L,\|u\|_{2,\Omega},p^{\pm},\partial\Omega)$ such that
\begin{equation}
\label{eq:p-est-1}
\begin{split}
\int_{\Omega}\left(\operatorname{div}\mathcal{F}_\epsilon(z,\nabla u)
\right)^2\,dx & \geq C\int_{\Omega}(\epsilon^2+|\nabla u|^2)^{p-2}|u_{xx}|^2\,dx
\\
&
- K\int_{\partial\Omega}(\mathcal{F}_\epsilon(z,\nabla u)\cdot \vec\nu)^2
\,dS
- \sum_{i,j=1}^N\sum_{s=4}^6 \int_{\Omega} |\mathcal{J}_{ij}^{(s)}|\,dx,
\end{split}
\end{equation}
where the constant $K$ is defined in \eqref{eq:fund-form} by the second fundamental form of the surface $\partial\Omega$.
\end{lemma}
\begin{cor}
\label{cor:convex-1} Let $\Omega$ be a smooth convex domain. Then  the
boundary integral on the right-hand side of \eqref{eq:p-est-1} can be omitted
by virtue of \eqref{eq:double-convex}, and the constant $C$ in
\eqref{eq:p-est-1} is independent of $\partial\Omega$.
\end{cor}

Now we combine inequality \eqref{eq:p-est-1} with the following assertion.
\begin{lemma}[Lemma 4.5, \cite{Arora-Shm-RACSAM-2023}]
\label{le:racsam}
Let $\partial\Omega\in C^2$, $\eta\in C^{0,1}(\Omega)$ with the Lipschitz
constant $L$. Assume that $\eta^->\frac{2N}{N+2}$. For every $u\in
C^3(\Omega)\cap C^2(\overline{\Omega})$, any $\delta\in (0,1)$, and $r\in
\left(0,\frac{4}{N+2}\right)$ the inequality

\[
\int_{\Omega}(\epsilon^2+|\nabla u|^2)^{\frac{\eta+r-2}{2}}|\nabla u|^2\,dx\leq \delta \int_{\Omega}(\epsilon^2+|\nabla u|^{2})^{\frac{\eta-2}{2}}|u_{xx}|^2\,dx +C
\]
holds with a constant $C=C(r,N,\eta^\pm,L,\delta,\|u\|_{2,\Omega})$.
\end{lemma}

Set $\eta=2(p-1)$ and claim $\eta^->\frac{2N}{N+2}$, which is equivalent to
\begin{equation}
\notag
\label{eq:lower-bound-p} p>1+\frac{N}{N+2}=\frac{2(N+1)}{N+2}.
\end{equation}
By Lemma \ref{le:racsam},  for every $\delta>0$ and $r\in \left(0,\frac{4}{N+2}\right)$
\begin{equation}
\label{eq:interpol-1}
\int_{\Omega} (\epsilon^2+|\nabla u|^2)^{p-2+\frac{r}{2}}|\nabla u|^2\,dx\leq \delta \int_{\Omega} (\epsilon^2+|\nabla u|^2)^{p-2}|u_{xx}|^2\,dx +C_\delta.
\end{equation}

We will repeatedly use the following elementary inequality: for every $\gamma>0$, $\sigma>0$ there is a constant $C=C(\gamma,\sigma)$ such that

\begin{equation}
\label{eq:elem}
s^\gamma |\ln s|\leq C(1+s^{\gamma+\sigma})\quad \text{for $s>0$}.
\end{equation}
An immediate corollary of Lemma \ref{le:racsam} and inequality \eqref{eq:elem} is that for every $\lambda\in \left(0,\frac{2}{N+2}\right)$ and $\delta\in (0,1)$
\begin{equation}
\label{eq:interpol-2}
\begin{split}
\int_{\Omega} (\epsilon^2+|\nabla u|^2)^{\frac{p-2}{2}} & |\nabla u|^2|\ln (\epsilon^2+|\nabla u|^2)|\,dx
\\
& \leq C\left(1+\int_{\Omega} (\epsilon^2+|\nabla u|^2)^{\frac{p-2}{2}+\lambda}|\nabla u|^2\,dx\right)
\\
&
\leq \delta \int_{\Omega} (\epsilon^2+|\nabla u|^2)^{p-2}|u_{xx}|^2\,dx +C(\delta,\lambda).
\end{split}
\end{equation}

\begin{cor}
\label{cor:interpol-prim} Since $p^->\frac{2(N+1)}{N+2}$, then
$p(z)-2+\frac{2}{N+2}>0$ in $Q_T$. It follows that there is $r_\ast\in
\left(0,\frac{4}{N+2}\right)$ such that for every $r\in
\left(r^\ast,\frac{4}{N+2}\right)$ and every $\delta>0$
\begin{equation}
\label{eq:interpol-prim} \int_{\Omega} |\nabla u|^{2(p-1)+r}\,dx\leq \delta
\int_{\Omega} (\epsilon^2+|\nabla u|^2)^{p-2}|u_{xx}|^2\,dx +C_\delta,\quad
r\in \left(r_\ast,\frac{4}{N+2}\right).
\end{equation}
By Young's inequality, this inequality is extended to the rest of the interval
$\left(0,\frac{4}{N+2}\right)$.
\end{cor}

We proceed to estimate the last two terms on the right-hand side of \eqref{eq:p-est-1}.

\subsubsection{Estimates on the residual terms}

Since $|\vec \eta|<1$ and $|D_i p|\leq L$, the residual terms
$\mathcal{J}^{(4,5,6)}_{ij}$ are estimated as follows: by
\eqref{eq:interpol-1}- \eqref{eq:interpol-2} and the Young inequality
\begin{equation}
\label{eq:est-J-4-5}
\begin{split}
\int_\Omega & (\epsilon^2+|\nabla u|^2)^{p-2}\left(|\mathcal{J}^{(4)}_{ij}| +|\mathcal{J}^{(5)}_{ij}|\right)\,dx
\\
& \leq C\int_{\Omega}\left((\epsilon^2+|\nabla u|^2)^{\frac{p-2}{2}}|u_{xx}|\right) \left((\epsilon^2+|\nabla u|^2)^{\frac{p-1}{2}}|\ln (\epsilon^2+ |\nabla u|^2)|\right)\,dx
\\
&
\leq \delta \int_\Omega(\epsilon^2+|\nabla u|^2)^{p-2}|u_{xx}|^2+C\int_\Omega(\epsilon^2+|\nabla u|^2)^{p-1+\lambda}\,dx+C',
\end{split}
\end{equation}
and
\begin{equation}
\label{eq:est-J-6}
\begin{split}
(\epsilon^2+|\nabla u|^2)^{p-2}|\mathcal{J}^{(6)}_{ij}| & \leq C(\epsilon^2+|\nabla u|^2)^{p-1}\ln^2(\epsilon^2+|\nabla u|^2)
\\
&
\leq C_1(\epsilon^2+|\nabla u|^2)^{p-1+\lambda}.
\end{split}
\end{equation}
The last term on the right-hand side of \eqref{eq:est-J-6} is estimated by means of the following inequality: for every $\lambda>0$
\begin{equation}
\label{eq:p-r}
\begin{split}
(\epsilon^2+|\nabla u|^2)^{p-1+\lambda} & \leq \begin{cases}
(2 \epsilon)^{2(p-1+ \lambda)} & \text{if $|\nabla u|\leq \epsilon$},
\\
(\epsilon^2+|\nabla u|^2)^{p-2+\lambda}|\nabla u|^2 & \text{if $|\nabla u|>\epsilon$}
\end{cases}
\\
&
\leq C\left(1+(\epsilon^2+|\nabla u|^2)^{p-2+\lambda}|\nabla u|^2 \right).
\end{split}
\end{equation}

Plugging \eqref{eq:p-r} into \eqref{eq:est-J-6}, and then using \eqref{eq:interpol-1}  with $\frac{r}{2}=\lambda$ we obtain: for every $\delta>0$
\begin{equation}
\label{eq:est-J-6-fin}
\int_{\Omega} (\epsilon^2+|\nabla u|^2)^{p-2}|\mathcal{J}^{(6)}_{ij}|\,dx\leq \delta
\int_{\Omega}(\epsilon^2+|\nabla u|^2)^{p-2}|u_{xx}|^2\,dx+C.
\end{equation}
\subsubsection{The boundary integrals}
Since $\partial\Omega\in C^2$, then $|\operatorname{trace}\mathcal{B}|\leq K$ with a finite constant $K>0$. Thus, estimating the boundary integral in \eqref{eq:double-final-e} amounts to estimating the integral of $(\epsilon^2+|\nabla u|^2)^{p-2}|\nabla u|^2$ over $\partial\Omega$.   By \cite[Lemma 1.5.1.9]{Grisvard} there is a function $\vec \mu\in C^{\infty}(\overline{\Omega})^N$ such that $\vec \mu\cdot\nu \geq \delta>0$ on $\partial\Omega$. Then
\[
\begin{split}
\delta\int_{\partial\Omega} & (\epsilon^2+|\nabla u|^2)^{p-2}|\nabla u|^2\,dS  \leq \int_{\Omega}\operatorname{div}\left((\epsilon^2+|\nabla u|^2)^{p-2}|\nabla u|^2\vec \mu\right)\,dx
 \\
 &
 = \int_{\Omega} \vec \mu \cdot\nabla ((\epsilon^2+|\nabla u|^2)^{p-2}|\nabla u|^2)\,dx + \int_{\Omega}(\operatorname{div}\vec \mu )(\epsilon^2+|\nabla u|^2)^{p-2}|\nabla u|^2\,dx
\\
& \leq C_1\int_{\Omega}(\epsilon^2+|\nabla u|^2)^{p-2}|\nabla u|^2\,dx
\\
&
+ C_2\int_{\Omega} \left[(\epsilon^2+|\nabla u|^2)^{p-3}|\nabla u|^2+(\epsilon^2+|\nabla u|^2)^{p-2}\right]|\nabla u||u_{xx}| \,dx
\\
&
+ C_3\int_{\Omega} (\epsilon^2+|\nabla u|^2)^{p-2}|\nabla u|^2 |\ln (\epsilon^2+|\nabla u|^2)||\nabla p|\,dx
\\
& \equiv \mathcal{I}_1 +\mathcal{I}_2 + \mathcal{I}_3.
\end{split}
\]
The integrals $\mathcal{I}_1$, $\mathcal{I}_3$ are estimated as $\mathcal{J}_{ij}^{(6)}$ in \eqref{eq:est-J-6} and \eqref{eq:p-r}: for every $\delta\in (0,1)$
\begin{equation}
\label{eq:I-1-3}
|\mathcal{I}_1| + |\mathcal{I}_3| \leq \delta \int_\Omega(\epsilon^2+|\nabla u|^2)^{p-2}|u_{xx}|^2\,dx+C_\delta.
\end{equation}
By Young's inequality
\[
\begin{split}
|\mathcal{I}_2| & \leq C \int_{\Omega} \left((\epsilon^2+|\nabla u|^2)^{\frac{p-2}{2}}|\nabla u|\right)\left((\epsilon^2+|\nabla u|^2)^{\frac{p-2}{2}}|u_{xx}|\right) \,dx
\\
& \leq \delta \int_\Omega (\epsilon^2+|\nabla u|^2)^{p-2}|u_{xx}|^2 \,dx + C \mathcal{I}_1,
\end{split}
\]
where $\mathcal{I}_1$ is already estimated in \eqref{eq:I-1-3}.
Gathering these estimates with a sufficiently small $\delta$, we reformulate Lemma \ref{le:principal-e} as follows.

\begin{lemma}
\label{le:principal-improved}
Let $\partial\Omega\in C^2$, $u\in C^3(\Omega)\cap C^{2}(\overline{\Omega})$. Assume that $p\in
C^2(\Omega)\cap C^{0,1}(\overline\Omega)$ with the Lipschitz constant $L$ and
\[
p^->\frac{2(N+1)}{N+2}.
\]
Then there exist constants $C_i=C_i(N,p^\pm,L,\partial\Omega, \|u\|_{2,\Omega})>0$, $i=1,2$, such that
\begin{equation}
\label{eq:p-est-2} C_1\int_{\Omega}(\epsilon^2+|\nabla
u|^2)^{p-2}|u_{xx}|^2\,dx \leq
\int_{\Omega}\left(\operatorname{div}\mathcal{F}_\epsilon (z,\nabla u)
\right)^2\,dx +C_2
\end{equation}
and, by Corollary \ref{cor:interpol-prim},
\begin{equation}
\label{eq:p-est-3}
\begin{split}
\int_\Omega |\nabla u|^{2(p-1)+r}\,dx & \leq \delta
\int_{\Omega}(\epsilon^2+|\nabla u|^2)^{p-2}|u_{xx}|^2\,dx +C
\\
& \leq \delta \int_{\Omega}(\operatorname{div}\mathcal{F}_\epsilon(z,\nabla
u))^2\,dx+C'
\end{split}
\end{equation}
where $C$, $C'$ depend on the same quantities as the constants in
\eqref{eq:p-est-2} but do not depend on $\epsilon$. By Corollary
\ref{cor:convex-1}, for convex domains $\Omega$ the constants $C_i$ in
\eqref{eq:p-est-2} and $C'$ in \eqref{eq:p-est-3} are independent of
$\partial\Omega$.
\end{lemma}

\section{A priori estimates}
\label{sec:estimates} Throughout this section, we assume that $u$ is the
strong solution of regularized problem \eqref{eq:main-reg-1} with data $f_m$,
$u_{0m}$ and $\textbf{data}_m$. The first a priori estimate is already derived
in Lemma \ref{le:first-est} for $\tau=1$. To obtain the second a priori
estimate we make use of the representation
\begin{equation}
\label{eq:u-t}
\begin{split}
& (\epsilon^2 +|\nabla u|^2)^{\frac{p_m-2}{2}}\nabla u\cdot \nabla u_t = \frac{1}{2}(\epsilon^2+|\nabla u|^2)^{\frac{p_m-2}{2}}\left(\epsilon^2+|\nabla u|^2\right)_t
\\
& = \frac{1}{2}\left(\frac{2}{p_m}\left(\epsilon^2+|\nabla u|^2\right)^{\frac{p_m}{2}}\right)_t
\\
&
\qquad - \left(\frac{1}{p_m}\right)_t\left(\epsilon^2+|\nabla u|^2\right)^{\frac{p_m}{2}} - \frac{1}{2p_m}(\epsilon^2 +|\nabla u|^2)^\frac{p_m}{2}\ln (\epsilon^2 +|\nabla u|^2) (p_m)_t
\\
& =\left(\frac{1}{p_m}(\epsilon^2 +|\nabla u|^2)^{\frac{p_m}{2}}\right)_t +\frac{(p_m)_t}{p_m}(\epsilon^2 +|\nabla u|^2)^{\frac{p_m}{2}}\left(\frac{1}{p_m}-\frac{1}{2}\ln (\epsilon^2 +|\nabla u|^2)\right).
\end{split}
\end{equation}
Now we multiply \eqref{eq:main-reg-1} by $-\operatorname{div}\mathcal{F}_{\epsilon m}(z,\nabla u)$, integrate over $Q_t=\Omega\times (0,t)$, use the condition $u=0$ on $\partial \Omega \times (0,T)$ and apply \eqref{eq:u-t} and the Cauchy inequality:
\begin{equation}
\label{eq:second-est}
\begin{split}
& \int_{\Omega}\frac{1}{p_m}(\epsilon^2 +|\nabla u|^2)^{\frac{p_m}{2}}\,dx + C\int_{Q_t}\left(\operatorname{div}\mathcal{F}_{\epsilon m}(z,\nabla u)\right)^2\,dz
\\
&
\leq C + \int_{\Omega}\frac{1}{p_m(x,0)}(\epsilon^2 +|\nabla u_{0m}|^2)^{\frac{p_m(x,0)}{2}}\,dx
+ \int_{Q_t}|f_m||\operatorname{div}\mathcal{F}_{\epsilon m}(z,\nabla u)|\,dz
\\
&
+M \int_{Q_t}(\epsilon^2 +|\nabla u|^2)^{\frac{p_m}{2}}\left(1+|\ln (\epsilon^2 +|\nabla u|^2)|\right)\,dz
+ \int_{Q_t} |F_{\epsilon m}(z,u,\nabla u)| |\operatorname{div}\mathcal{F}_{\epsilon m}(z,\nabla u)|\,dz
\\
& \leq C + \int_{\Omega}\frac{1}{p_m(x,0)}(\epsilon^2 +|\nabla u_{0m}|^2)^{\frac{p_m(x,0)}{2}}\,dx + C_\delta \int_{Q_t}f_m^2\,dz +\delta \int_{Q_t}\left(\operatorname{div}\mathcal{F}_{\epsilon m}(z,\nabla u)\right)^2\,dz
\\
& \qquad + C'\int_{Q_t} (\epsilon^2+u^2)^{q_m-1}\,dz
+ C''\int_{Q_t} (\epsilon^2+|\nabla u|^{2})^{s_m-1}\,dz
\\
&
\qquad +
M \int_{Q_t}(\epsilon^2 +|\nabla u|^2)^{\frac{p_m}{2}}\left(1+|\ln (\epsilon^2 +|\nabla u|^2)|\right)\,dz,\quad \delta\in (0,1).
\end{split}
\end{equation}
The constant $M$ depends on the Lipschitz constant $L$ of $p_m$.
The last term on the right-hand side of \eqref{eq:second-est} is estimated by virtue of \eqref{eq:interpol-1}, \eqref{eq:interpol-2}: it is sufficient to claim that
\[
p_m<2(p_m-1)+\frac{4}{N+2} \quad \Leftrightarrow \quad p_m>2-\frac{4}{N+2}=\frac{2N}{N+2},
\]
while $p_m^->\frac{2(N+1)}{N+2}$ by assumption. Using Lemma \ref{le:principal-improved} and choosing $\delta>0$ sufficiently small we obtain

\begin{equation}
\label{eq:2-nd-est-1}
\begin{split}
\sup_{(0,T)}
 &
 \int_{\Omega}(\epsilon^2 +|\nabla u|^2)^{\frac{p_m}{2}}\,dx+C\int_{Q_T}(\epsilon^2+|\nabla u|^2)^{p_m-2}|u_{xx}|^2\,dz
\\
& \leq \sup_{(0,T)} \int_{\Omega}(\epsilon^2 +|\nabla u|^2)^{\frac{p_m}{2}}\,dx+\int_{Q_T} \left(\operatorname{div}\mathcal{F}_{\epsilon m} (z,\nabla u)\right)^2\,dz
\\
&
 \leq C'\left(1+\int_{\Omega}\frac{1}{p_m(x,0)}(\epsilon^2 +|\nabla u_{0m}|^2)^{\frac{p_m(x,0)}{2}}\,dx + \int_{Q_t}f_m^2\,dz\right)
\\
& \qquad + C'+C''\int_{Q_t} u^{2(q_m-1)}\,dz
+ C'''\int_{Q_t} |\nabla u|^{2(s_m-1)}\,dz.
\end{split}
\end{equation}
If $q_m^+\leq 2$, then $u^{2(q_m-1)}\leq C(1+u^2)$. By Lemma \ref{le:first-est} with $\tau=1$
\[
\int_{Q_T}u^{2(q_m-1)}\,dz\leq C\int_0^T(1+\|u\|_{2,\Omega}^2)\,dt\leq C'T.
\]
If $q_m^+>2$, we have $u^{2(q_m-1)}\leq C(1+u^{2(q_m^+-1)})$. By the Gagliardo-Nirenberg inequality

\begin{equation}
\label{eq:G-N}
\int_{\Omega}|u|^{2(q_m^+-1)}\,dx\leq C \|\nabla u\|_{p_m^-,\Omega}^{2\theta(q_m^+-1)} \|u\|_{2,\Omega}^{2(1-\theta)(q_m^+-1)}+C\|u\|_{2,\Omega}^{2(q_m^+-1)}
\end{equation}
with $\theta= \frac{\frac{1}{2}-\frac{1}{2(q_m^+-1)}}{\frac{N+2}{2N}-\frac{1}{p_m^-}}\in [0,1]$. Since $q_m^+\geq 2$, such a $\theta$ exists if
\begin{equation}
\label{eq:ineq-1}
2(q_m^+-1)\leq \dfrac{Np_m^-}{N-p^-}\qquad \Leftrightarrow\quad q_m^+\leq 1+\dfrac{1}{2}\dfrac{Np_m^-}{N-p_m^-}\qquad \text{if $p_m^-
<N$}.
\end{equation}
Moreover, we have to claim that
\begin{equation}
\label{eq:ineq-2}
2\theta (q_m^+-1)\leq p_m^-\qquad \Leftrightarrow \qquad q_m^+\leq 1+p_m^-\frac{N+2}{2N}.
\end{equation}
Inequality \eqref{eq:ineq-2} is stronger than \eqref{eq:ineq-1}. Integrating \eqref{eq:G-N} in $t$ and using \eqref{eq:ineq-2} we obtain the inequality
\begin{equation}
\label{eq:q-est}
\begin{split}
\int_{Q_T}u^{2(q_m-1)}\,dz
\leq  C'+C''\int_0^T \left(\int_\Omega|\nabla u|^{p_m^-}\,dx\right)^{\frac{2\theta(q_m^+-1)}{p_m^-} }\,dt.
\end{split}
\end{equation}
Using condition \eqref{eq:ineq-2}, the H\"older and Young inequalities, by Lemma \ref{le:first-est} with $\tau=1$ we have
\[
\int_0^T \left(\int_\Omega|\nabla u|^{p_m^-}\,dx\right)^{\frac{2\theta(q_m^+-1)}{p_m^-} }\,dt \leq C'+C''\int_{Q_T}|\nabla u|^{p_m}\,dz \leq C,
\]
where the constant $C$ depends on $|\Omega|$, $T$, $\delta$ and $\|f_m\|_{2,Q_T}$, $\|u_{0m}\|_{2,\Omega}$. These arguments prove the following assertion.
\begin{proposition}
\label{pro:2-est-1}
Let the conditions of Lemma \ref{le:first-est} be fulfilled. If
\[
1<q_m(z)\leq 1+p_m^-\frac{N+2}{2N}\quad \text{and} \quad p_m^->\frac{2(N+1)}{N+2},
\]
then for the solution of problem \eqref{eq:main-reg-1}

\[
\int_{Q_T} a_m(z)(\epsilon^2+u^2)^{\frac{q_m-2}{2}}u\operatorname{div}\mathcal{F}_{\epsilon m}(z,\nabla u)\,dz \leq \delta \int_{Q_T}(\operatorname{div}\mathcal{F}_{\epsilon m}(z,\nabla u))^2\,dz + C
\]
with an arbitrary positive $\delta$ and a constant $C$ depending only $\|f_m\|_{2,Q_T}$, $\|u_{0m}\|_{2,\Omega}$, $T$, $|\Omega|$, $a^+$, $\delta$.
\end{proposition}

\begin{proposition}
\label{pro:2-est-2}
Assume that the conditions of Lemma \ref{le:first-est} are fulfilled and  $p_m^->\frac{2(N+1)}{N+2}$. Then
\begin{equation}
\label{eq:2-est-2}
\int_{Q_T} (\epsilon^2+|\nabla u|^{2})^{\frac{s_m-2}{2}}(\vec c_m,\nabla u)\operatorname{div}\mathcal{F}_{\epsilon m}(z,\nabla u)\,dx\leq \delta \int_{Q_T}(\operatorname{div}\mathcal{F}_{\epsilon m}(z,\nabla u))^2\,dz +C
\end{equation}
with any $\delta>0$ and a constant $C$ depending on $p^\pm$, $|\Omega|$, $T$, $\max c_i^+$, $\delta$ and $\operatorname{ess}\sup_{(0,T)}\|u\|_{2,\Omega}$.
\end{proposition}

\begin{proof}
By Young's inequality, \eqref{eq:interpol-1} and Lemma \ref{le:principal-improved}
\[
\begin{split}
\int_{\Omega} & (\epsilon^2+|\nabla u|^{2})^{\frac{s_m-2}{2}}(\vec c_m,\nabla u)\operatorname{div}\mathcal{F}_{\epsilon m}(z,\nabla u)\,dx
\\
& \leq \delta\int_\Omega (\operatorname{div}\mathcal{F}_{\epsilon m}(z,\nabla u))^2\,dx
 +C\int_\Omega (\epsilon^2+|\nabla u|^2)^{s_{m}-2}|\nabla u|^2
\,dx
\\
& \leq \delta\int_\Omega (\operatorname{div}\mathcal{F}_{\epsilon m}(z,\nabla u))^2\,dx +\delta \int_\Omega (\epsilon^2+|\nabla u|^2)^{p_m-2}|u_{xx}|^2\,dx+ C.
\end{split}
\]
\end{proof}
To estimate $u_t$ we use \eqref{eq:main-reg-1} and Propositions \ref{pro:2-est-1}, \ref{pro:2-est-2}:
\begin{equation}
\label{eq:der-t}
\begin{split}
& \|u_t\|_{2,Q_T}^2 =\int_{Q_T}\left(f_m+ F_{\epsilon m}(z,u,\nabla u)+\operatorname{div}((\epsilon^2 +|\nabla u|^2)^{\frac{p_m-2}{2}}\nabla u)\right)^2\,dz
\\
&
\quad \leq 3\|f_m\|_{2,Q_T}^2 +3\int_{Q_T}\left(\operatorname{div}\mathcal{F}_{\epsilon m}(z,\nabla u)\right)^2\,dz
\\
& \qquad + 6\int_{Q_T}a_m^2(\epsilon^2+u^2)^{q_m-2}u^2\,dz +6 \int_{Q_T}|c_m|^2(\epsilon^2+|\nabla u|^2)^{s_m-2}|\nabla u|^2\,dz
\\
& \quad \leq C\left(1+\int_{\Omega}\frac{1}{p_m(x,0)}(\epsilon^2 +|\nabla u_{0m}|^2)^{\frac{p_m(x,0)}{2}}\,dx + \int_{Q_t}f_m^2\,dz\right)
\end{split}
\end{equation}
with an independent of $u$ constant $C$. Gathering these estimates and recalling \eqref{eq:data-reg}-\eqref{eq:data-reg-0} we obtain the following assertion.

\begin{lemma}
\label{le:second-estimate} Let conditions \eqref{eq:p}-\eqref{eq:data} be
filfilled and $\partial\Omega\in C^{2+\alpha}$. Then the strong solution of
problem \eqref{eq:main-reg-1} with the data $f_m$, $u_{0m}$ and
$\textbf{data}_m$ satisfies the estimate
\[
\begin{split}
\|u_t\|^2_{2,Q_T} & + \sup_{(0,T)}
 \int_{\Omega}(\epsilon^2 +|\nabla u|^2)^{\frac{p_m}{2}}\,dx +\int_{Q_T}(\epsilon^2+|\nabla u|^2)^{p_m-2}|u_{xx}|^2\,dz
 \\
 & \leq \|u_t\|^2_{2,Q_T} + \sup_{(0,T)}
 \int_{\Omega}(\epsilon^2 +|\nabla u|^2)^{\frac{p_m}{2}}\,dx +C' \int_{Q_T} (\operatorname{div}\mathcal{F}_{\epsilon m}(z,\nabla u))^2\,dz +C
 \\
 &
 \leq C'\left(1+\int_\Omega |\nabla u_{0m}|^{p_m(x,0)}\,dx +\|f_m\|_{2,Q_T}^2\right)
 \\
 &
 \leq C''\left(1+\int_\Omega |\nabla u_{0}|^{p_0}\,dx +\|f\|_{2,Q_T}^2\right)
\end{split}
\]
with an independent of $\epsilon$ and $m$ constant $C''$.
\end{lemma}

\section{Existence and regularity of strong solutions}
\label{sec:proofs}
\subsection{The regularized problem}

We begin with considering the regularized problem \eqref{eq:main-reg}. Let
$u_m$ be the solution of \eqref{eq:main-reg-1} with the free term $f_m$ and
the initial datum $u_{0m}$ satisfying \eqref{eq:data-reg}, and
$\textbf{data}_m$. By Lemma \ref{le:first-est} and Lemma
\ref{le:second-estimate}

\begin{equation}
\label{eq:unif-est}
\begin{split}
\operatorname{ess} & \sup_{(0,T)}\|u_{m}\|_{2,\Omega} +\|u_{m t}\|_{2,Q_T}+ \operatorname{ess}\sup_{(0,T)}\|\nabla u_m\|_{p_m(\cdot),\Omega}
\\
&
+\|(\epsilon^2+|\nabla u_m|^2)^{\frac{p_m-2}{2}}|(u_m)_{xx}| \|_{2,Q_T}+\|\operatorname{div}\mathcal{F}_{\epsilon m}(z,\nabla u_m) \|_{2,Q_T}\leq C
\end{split}
\end{equation}
with a constant $C$ independent of $m$ and $\epsilon$. By virtue of \eqref{eq:interpol-1}, for any $r\in \left(0,\frac{4}{N+2}\right)$
\begin{equation}
\label{eq:high-1}
\|(\epsilon^2+|\nabla u_m|^2)^{\frac{p_m-2}{2}+\frac{r}{4}}|\nabla u_m|\|_{2,Q_T} + \||\nabla u_m|\|_{p_m(\cdot),Q_T} \leq C,
\end{equation}
which yields the uniform in $m$ and $\epsilon$ estimate
\begin{equation}
\label{eq:high-2}
\|(\epsilon^2+|\nabla u_m|^2)^{\frac{s_m-2}{2}}|\nabla u_m|\|_{2+\lambda,Q_T}\leq C
\end{equation}
with some positive $\lambda$. It follows from estimate \eqref{eq:q-est} and
Proposition \ref{pro:extra-int} that
\begin{equation}
\label{eq:high-3}
\|(\epsilon^2+u_m^2)^{\frac{q_m-2}{2}}|u_m|\|_{2+\lambda,Q_T}\leq C
\end{equation}
with a positive constant $\lambda$ and an indepepdent of $m$, $\epsilon$ constant $C$. Moreover, the convergence $p_m \nearrow p$ in $C^{0,1}(\overline{Q}_T)$ and \eqref{eq:high-1} imply that for all sufficiently large $m$

\begin{equation}
\label{eq:high-4}
 \||\nabla u_m|\|_{p(\cdot),Q_T} + \||\mathcal{F}_{\epsilon m}(z,u_m,\nabla u_m)|\|_{p'(\cdot), Q_T} \leq C
\end{equation}
with an independent of $m$ and $\epsilon$ constant $C$. It follows that the sequence $\{u_m\}$ contains a subsequence (for which we will keep the same notation), and there exist functions $u$, $\chi$, $\eta$, $\zeta$  such that
\begin{equation}
\label{eq:conv-1}
\begin{split}
& \text{$u_m\to u$ $\star$-weakly in $L^\infty(0,T;L^2(\Omega))$, strongly in $L^2(Q_T)$ and a.e. in $Q_T$},
\\
& \text{$u_{m t}\rightharpoonup u_t$ in $L^2(Q_T)$},\quad
\text{$\nabla u_m \rightharpoonup \nabla u$ in $L^{p(\cdot)}(Q_T)^N$},
\\
& \text{$\mathcal{F}_{\epsilon m} (z,\nabla u_m)\rightharpoonup \eta$ in $L^{p'(\cdot)}(Q_T)^N$},
\\
& \text{$\operatorname{div}\mathcal{F}_{\epsilon m}(z,\nabla u_m)\rightharpoonup \zeta$ in $L^2(Q_T)$},
\\
& \text{$F_{\epsilon m}(z,u_m,\nabla u_m)\rightharpoonup \chi$ in $L^2(Q_T)$}.
\end{split}
\end{equation}

\begin{lemma}
\label{le:strong-conv-grad}
The sequence $\{\nabla u_m\}$ converges to $\nabla u$ in $(L^{p(\cdot)}(Q_T))^N$.
\end{lemma}

\begin{proof}
Let us multiply equations \eqref{eq:main-reg-1} for $u_m$ and $u_n$ by $w=u_m-u_n$, integrate over $\Omega$ and combine the results:
\begin{equation}
\label{eq:comb-1}
\begin{split}
\frac{1}{2} & \dfrac{d}{dt}\|w\|^2_{2,\Omega} +\int_{\Omega}(\mathcal{F}_{\epsilon m}(z,\nabla u_m)-\mathcal{F}_{\epsilon n}(z,\nabla u_n))\cdot \nabla w\,dx
\\
& = \int_{\Omega}(f_m-f_n)w\,dz
 +\int_{\Omega}\left[a_m(\epsilon^2+ u_m^2)^{\frac{q_m-2}{2}} u_m- a_n(\epsilon^2+ u_n^2)^{\frac{q_n-2}{2}} u_n\right] w \,dz
\\
& + \int_{\Omega}\left[(\epsilon^2+ |\nabla u_m|^2)^{\frac{s_m-2}{2}}(\vec c_m,\nabla u_m)- (\epsilon^2+ |\nabla u_n|^2)^{\frac{s_n-2}{2}}(\vec c_n,\nabla u_n)\right] w \,dz.
\end{split}
\end{equation}
Let us integrate this inequality over the interval $(0,T)$.
Applying the H\"older inequality and \eqref{eq:high-2}, \eqref{eq:high-3}, we
estimate the last two terms of the resulting inequality by $M\|w\|_{2,Q_T}$
with an independent of $m$, $n$ constant $M$. Inequality \eqref{eq:comb-1} is
then continued as follows:
\begin{equation}\label{eq:intermediate}
\begin{split}
\frac{1}{2}& \|w(T)\|^{2}_{2,\Omega}  +\int_{Q_T}(
\mathcal{F}_{\epsilon m}(z,\nabla u_m)-\mathcal{F}_{\epsilon m}(z,\nabla
u_n))\cdot \nabla w\,dz
\\
&
\leq \frac{1}{2}\|w(\cdot,0)\|_{2,\Omega}^2 + \|w\|_{2,Q_T}\|f_m-f_n\|_{2,Q_T}
+M\sqrt{T}\|w\|_{2,Q_T} \\ & \qquad + \int_{Q_T}(
\mathcal{F}_{\epsilon n}(z,\nabla u_n)-\mathcal{F}_{\epsilon m}(z,\nabla
u_n))\cdot \nabla w\,dz\\
&  \leq \frac{1}{2}\|w(\cdot,0)\|_{2,\Omega}^2 +
\|w\|_{2,Q_T}\|f_m-f_n\|_{2,Q_T}
\\
&
+M\sqrt{T}\|w\|_{2,Q_T} + C\sup_{Q_T} |p_m(z)
- p_n(z)|^{p_m'(z)}.
\end{split}
\end{equation}
To estimate the last term of the second inequality in
\eqref{eq:intermediate} we apply the Lagrange mean value
theorem (see the derivation of \eqref{eq:flux-cont}) and then follow the
proof of \cite[Lemma 3.9]{SSS-2022}. The constant $C$ depends on $\|\nabla
u_m\|_{p(\cdot),Q_T}$, $\|\nabla u_n\|_{p(\cdot),Q_T}$, which are already
estimated in \eqref{eq:high-4}.

By the choice of the sequences $\{f_m\}$, $\{u_{0m}\}$ and due to \eqref{eq:conv-1}, the right-hand side of inequality \eqref{eq:intermediate} tends to zero as $m,n\to \infty$. Let us denote
\begin{equation}
\label{eq:R}
\mathcal{R}_{\epsilon m}(\nabla u_m,\nabla u_n)= \int_{Q_T}(\mathcal{F}_{\epsilon m}(z,\nabla u_m)-\mathcal{F}_{\epsilon m}(z,\nabla u_n))\cdot \nabla w\,dz.
\end{equation}
Dropping the first nonnegative term on the left-hand side of \eqref{eq:intermediate} and taking into account \eqref{eq:mon-reg} we conclude that
\begin{equation}
\notag
\label{eq:F-limit}
0\leq \mathcal{R}_{\epsilon m}(\nabla u_m,\nabla u_n)\to 0\quad \text{as $m,n\to \infty$}.
\end{equation}
By \cite[Proposition 3.2]{Arora-Shm-RACSAM-2023} with $a(z)\equiv 1$, $b(z)\equiv 0$ and $\epsilon=0$ (see also \cite[Lemma 3.8]{SSS-2022}) we know that
\begin{equation}
\label{eq:R-conv}
\int_{Q_T}|\nabla (u_m-u_n)|^{p_m(z)}\,dz\leq C\left(\mathcal{R}_{\epsilon m}(\nabla u_m,\nabla u_n)+\mathcal{R}_{\epsilon m}^{\frac{p_m^-}{2}}(\nabla u_m,\nabla u_n)+\mathcal{R}_{\epsilon m}^{\frac{p_m^+}{2}}(\nabla u_m,\nabla u_n)\right)
\end{equation}
with a constant $C$ depending on $p^\pm$, $\|\nabla u_m\|_{p(\cdot),Q_T}$, $\|\nabla u_n\|_{p(\cdot),Q_T}$. Since the sequence $\{p_k\}$ is nondecreasing, we have $p_1 \leq p_m$ for all $m\in \mathbb{N}$. It follows that $\{\nabla u_m\}$ is a Cauchy sequence in $L^{p_1(\cdot)}(Q_T)^N$, and $\nabla u_m \to \nabla u$ a.e. in $\Omega$. Finally, the Vitali convergence theorem along with the uniform estimates in \eqref{eq:high-1} and \eqref{eq:high-4}, implies $\nabla u_m \to \nabla u$ in $(L^{p(\cdot)}(Q_T))^N.$
\end{proof}

For every test-function $\phi\in  \mathbf{W}_{p(\cdot)}(Q_T)$

\begin{equation}
\label{eq:def-eps}
\int_{Q_T}\left(u_{m t}\phi +\mathcal{F}_{\epsilon m}(z,\nabla u_m)\cdot \nabla \phi - f_m\phi - F_{\epsilon m}(z,u_m,\nabla u_m)\phi\right)\,dz=0.
\end{equation}
Lemma \ref{le:strong-conv-grad} and estimates \eqref{eq:high-1}, \eqref{eq:high-2} allow one to apply the Vitali convergence theorem:
\[
\begin{split}
& (\epsilon^2+|\nabla u_m|^2)^{\frac{s_m-2}{2}}\nabla u_m\to (\epsilon^2+|\nabla u|^2)^{\frac{s-2}{2}}\nabla u \quad \text{in $L^2(Q_T)^N$},
\\
& \mathcal{F}_{\epsilon m}(z,\nabla u_m)\to \mathcal{F}_{\epsilon}(z,\nabla u)\quad \text{in $L^{p'(\cdot)}(Q_T)^N$},
\end{split}
\]
which means that $\eta=\mathcal{F}_{\epsilon}(z,\nabla u)$ a.e. in $Q_T$. The a.e. convergence $u_m\to u$ and estimate \eqref{eq:high-3} imply
\[
(\epsilon^2+u_m^2)^{\frac{q_m-2}{2}}|u_m|\to (\epsilon^2+u^2)^{\frac{q-2}{2}}|u|\quad \text{in $L^2(Q_T)$}.
\]
Passing to the limit in \eqref{eq:def-eps} as $m\to \infty$ we obtain \eqref{eq:def-1} for the function $u(z)=\lim u_m(z)$. The initial condition is fulfilled by continuity because $u\in C^0([0,T];L^2(\Omega))$ by the Aubin-Lions-Simon lemma.

\begin{theorem}
\label{th:exist-reg} Let conditions \eqref{eq:p}-\eqref{eq:data} be fulfilled,
and $\partial\Omega\in C^{2+\alpha}$. For every $f\in L^{2}(Q_T)$, $u_0\in
W^{1,p_0(\cdot)}_0(\Omega)$, and $\epsilon\in (0,1)$ problem
\eqref{eq:main-reg} has a strong solution. The solution satisfies the estimate

\begin{equation}
\label{eq:unif-eps}
\begin{split}
\operatorname{ess}\sup_{(0,T)}\|u\|_{2,\Omega} & + \operatorname{ess}\sup_{(0,T)}\|\nabla u\|_{p(\cdot),\Omega}+\|u_{t}\|_{2,Q_T}
\leq C
\end{split}
\end{equation}
and possesses the property of global higher integrability of the gradient: for every $r\in \left(0,\frac{4}{N+2}\right)$

\begin{equation}
\label{eq:unif-high} \int_{Q_T}|\nabla u|^{2(p(z)-1)+r}\,dz\leq
\int_{Q_T}(\epsilon^2+|\nabla u|^2)^{p(z)-2+\frac{r}{2}}|\nabla u|^2\,dz\leq
C'.
\end{equation}
The constant $C$ depends on $N$, $\partial \Omega$, $L$, the structural constants in \eqref{eq:p}-\eqref{eq:data}, $\|f\|_{2,Q_T}$, $\|\nabla u_0\|_{p_0(\cdot),\Omega}$,   the constant $C'$ depends on the same quantities and
also on $r$; $C$, $C'$ are independent of $\epsilon$.
\end{theorem}

\begin{theorem}
\label{th:reg-eps} Let the conditions of Theorem \ref{th:exist-reg} be
fulfilled. If $u$ is a strong solution of problem \eqref{eq:main-reg}, then
\[
(\epsilon^2+|\nabla u|^2)^{\frac{p(z)-2}{2}}D_i u\in W^{1,2}(Q_T),\quad
i=\overline{1,N},
\]
and
\begin{equation}
\label{eq:high-diff-0} \left\|(\epsilon^2+|\nabla u|^2)^{\frac{p(z)-2}{2}}D_i
u\right\|_{W^{1,2}(Q_T)}\leq C
\end{equation}
with $C$ depending on the same quantities as the constant $C'$ in \eqref{eq:unif-high} but independent of $\epsilon$.
\end{theorem}

\begin{proof}
Let $u_m$ be the solution of problem \eqref{eq:main-reg-1}. For all $i,j=\overline{1,N}$
\[
\begin{split}
D_j \left((\epsilon^2+|\nabla u_m|^2)^{\frac{p_m-2}{2}}D_i u_m\right) & = (\epsilon^2+|\nabla u_m|^2)^{\frac{p_m-2}{2}}D^2_{ij} u_m
\\
& + (p_m-2) (\epsilon^2+|\nabla u_m|^2)^{\frac{p_m-2}{2}-1}\sum_{k=1}^ND_k u_m D^{2}_{kj}u_m D_iu_m
\\
& + (\epsilon^2+|\nabla u_m|^2)^{\frac{p_m-2}{2}}D_iu_m \ln (\epsilon^2+|\nabla u_m|^2)D_j p_m.
\end{split}
\]
Using \eqref{eq:interpol-1}, \eqref{eq:interpol-2} and Lemma \ref{le:second-estimate} we conclude that
\begin{equation}
\label{eq:high-diff-1}
\begin{split}
& \left\|D_j \left((\epsilon^2+|\nabla u_m|^2)^{\frac{p_m-2}{2}}D_i u_m\right)\right\|_{2,Q_T}^2
\\
& \qquad \qquad
\leq C+C'\left\|(\epsilon^2+|\nabla u_m|^2)^{\frac{p_m-2}{2}}|(u_m)_{xx}|\right\|_{2,Q_T}^2\leq C''
\end{split}
\end{equation}
with a constant $C''$ depending only on the data and independent of $m$ and $\epsilon$. It follows that there is $\Theta_{ij}^{(\epsilon)}\in L^2(Q_T)$ and a subsequence of $\{u_m\}$ (for which we keep the same notation) such that
\begin{equation}
\label{eq:high-diff-2}
D_j \left((\epsilon^2+|\nabla u_m|^2)^{\frac{p_m-2}{2}}D_i u_m\right) \rightharpoonup \Theta_{ij}^{(\epsilon)} \quad \text{in $L^2(Q_T)$}.
\end{equation}
To identify the limit $\Theta_{ij}^{(\epsilon)}$ we use the pointwise
convergence of the sequence of gradients: for every $\psi\in
C_0^{\infty}(Q_T)$
\[
\begin{split}
\left(D_j\psi,(\epsilon^2+|\nabla u|^2)^{\frac{p-2}{2}}D_i u\right)_{2,Q_T} & = \lim_{m\to \infty}\left(D_j\psi,(\epsilon^2+|\nabla u_m|^2)^{\frac{p_m-2}{2}}D_i u_m\right)_{2,Q_T}
\\
& =- \lim_{m\to\infty}\left(\psi,D_j\left((\epsilon^2+|\nabla u_m|^2)^{\frac{p_m-2}{2}}D_i u_m\right)\right)_{2,Q_T}
\\
& = -(\psi,\Theta_{ij}^{(\epsilon)})_{2,Q_T}.
\end{split}
\]
Estimate \eqref{eq:high-diff-0} follows from \eqref{eq:high-diff-1}, \eqref{eq:high-diff-2}. Indeed, for every $i,j=\overline{1,N}$, we have
\[
\begin{split}
\|\Theta_{ij}^{(\epsilon)}\|^2_{L^2(\Omega)} & =\langle \Theta_{ij}^{(\epsilon)}, \Theta_{ij}^{(\epsilon)}\rangle_{2, \Omega}
\\
&
= \lim_{m \to \infty} \left\langle D_j \left((\epsilon^2+|\nabla u_m|^2)^{\frac{p_m-2}{2}}D_i u_m\right), \Theta_{ij}^{(\epsilon)}\right\rangle_{2, \Omega} \\
& \leq \|\Theta_{ij}^{(\epsilon)}\|_{L^2(\Omega)} \lim_{m \to \infty} \left\|D_j \left((\epsilon^2+|\nabla u_m|^2)^{\frac{p_m-2}{2}}D_i u_m\right)\right\|_{L^2(\Omega)}
\\
&
\leq C'' \|\Theta_{ij}^{(\epsilon)}\|_{L^2(\Omega)}
\end{split}
\]
which implies
\[
\|\Theta_{ij}^{(\epsilon)}\|_{L^2(\Omega)} \leq C''
\]
with a constant $C''$ independent of $\epsilon.$
\end{proof}

\begin{cor}
\label{cor:div-m}
The same arguments allow one to identify $\zeta$ in \eqref{eq:conv-1}: $\zeta=\operatorname{div}\mathcal{F}_\epsilon(z,\nabla u)$ and $\|\zeta\|_{2,Q_T}$ is bounded by the constant from \eqref{eq:unif-est}.
\end{cor}

\begin{cor}
\label{cor:strong-def} For every $\phi\in L^2(Q_T)$
\[
\int_{Q_T}\left(u_t-\operatorname{div}\mathcal{F}_\epsilon(z,\nabla u)-f-F_\epsilon(z,u,\nabla u)\right)\phi\,dz=0.
\]
\end{cor}

\subsection{The degenerate problem}
The existence and regularity results for the degenerate problem
\eqref{eq:main} are obtained by the arguments similar to the proofs of
Theorems \ref{th:exist-reg}, \ref{th:reg-eps}. Let us denote by
$u^{(\epsilon)}$ the strong solution of problem \eqref{eq:main-reg} with the
data $f$, $u_{0}$ constructed in Theorem \ref{th:exist-reg}.

\begin{proof}[Proof of Theorem \ref{th:existence-degenerate}]

The uniform estimates
\eqref{eq:unif-eps}, \eqref{eq:high-diff-0} imply that the family
$\{u^{(\epsilon)}\}$ contains a sequence $\{u^{(\epsilon_k)}\}$, for which we
will use the same notation  $\{u^{(\epsilon)}\}$, with the following convergence properties:
\begin{equation}
\label{eq:conv-2}
\begin{split}
& \text{$u^{(\epsilon)} \to u$ $\star$-weakly in $L^\infty(0,T;L^2(\Omega))$, strongly in
$L^2(Q_T)$ and a.e. in $Q_T$},
\\
& \text{$u^{(\epsilon)}_{t}\rightharpoonup u_t$ in $L^2(Q_T)$},\quad \text{$\nabla u_m
\rightharpoonup \nabla u$ in $L^{p(\cdot)}(Q_T)^N$},
\\
& \text{$\mathcal{F}_\epsilon (z,\nabla u^{(\epsilon)})\rightharpoonup \eta$ in
$L^{p'(\cdot)}(Q_T)^N$},
\\
& \text{$F_\epsilon(z,u^{(\epsilon)},\nabla u^{(\epsilon)})\rightharpoonup \chi$ in $L^2(Q_T)$}
\end{split}
\end{equation}
for some $u$, $\eta$, $\chi$. Combining equalities \eqref{eq:def-1} for $u^{(\epsilon)}$ with the test-functions $u^{(\epsilon)}$ and $\phi\in \mathbf{W}_{p(\cdot)}(Q_T)$ we obtain the equality
\[
\begin{split}
\int_{Q_T} & (\mathcal{F}_{\epsilon}(z,\nabla u^{(\epsilon)})-\mathcal{F}_{\epsilon}(z,\nabla \phi))\cdot \nabla (u^{(\epsilon)}-\phi)\,dz = -\int_{Q_T}u^{(\epsilon)}_{t}(u^{(\epsilon)}-\phi)\,dz
\\
&
-\int_{Q_T}(\mathcal{F}_{\epsilon}(z,\nabla \phi)-\mathcal{F}(z,\nabla \phi))\cdot \nabla (u^{(\epsilon)}-\phi)\,dz -\int_{Q_T} \mathcal{F}(z,\nabla \phi))\cdot \nabla (u^{(\epsilon)}-\phi)\,dz
\\
& + \int_{Q_T} f (u^{(\epsilon)}-\phi)\,dz +\int_{Q_T}F_{\epsilon}(z,u^{(\epsilon)},\nabla u^{(\epsilon)}) (u^{(\epsilon)}-\phi)\,dz
\end{split}
\]
whence, by monotonicity of the flux $\mathcal{F}_{\epsilon}$,
\[
\begin{split}
0 & \leq -\int_{Q_T}u^{(\epsilon)}_{t}(u^{(\epsilon)}-\phi)\,dz -\int_{Q_T}(\mathcal{F}_{\epsilon}(z,\nabla \phi)-\mathcal{F}(z,\nabla \phi))\cdot \nabla (u^{(\epsilon)}-\phi)\,dz
\\
&
-\int_{Q_T} \mathcal{F}(z,\nabla \phi))\cdot \nabla (u^{(\epsilon)}-\phi)\,dz
\\
& +\int_{Q_T} f (u^{(\epsilon)}-\phi)\,dz +\int_{Q_T} F_{\epsilon}(z,u^{(\epsilon)},\nabla u^{(\epsilon)}) (u^{(\epsilon)}-\phi)\,dz
= \sum_{i=1}^5\mathcal{I}_i.
\end{split}
\]
Each of $\mathcal{I}_i$ has a limit as $\epsilon\to 0$.  As in the case of the regularized problem,
\[
\mathcal{I}_1\to \displaystyle -\int_{Q_T}u_t(u-\phi)\,dz,\qquad \mathcal{I}_5\to \int_{Q_T}\chi (u-\phi)\,dz
\]
as the limits of weakly and strongly convergent sequences,
\[
\mathcal{I}_3\to \displaystyle -\int_{Q_T}\mathcal{F}(z,\nabla \phi)\cdot \nabla (u-\phi)\,dz
\]
because $\mathcal{F}(z,\nabla \phi)\in L^{p'(\cdot)}(Q_T)^N$ and $\nabla u^{(\epsilon)}\rightharpoonup \nabla u$ in $L^{p(\cdot)}(Q_T)^N$,
\[
\mathcal{I}_4\to \displaystyle \int_{Q_T} f (u-\phi)\,dz.
\]
To estimate $\mathcal{I}_2$ we apply the generalized H\"older inequality \eqref{eq:gen-Holder}
\[
\begin{split}
\mathcal{I}_2 & \leq 2\left(\|\nabla u^{(\epsilon)}\|_{p(\cdot),Q_T}+\|\nabla \phi \|_{p(\cdot),Q_T}\right)\left\|\mathcal{F}_{\epsilon}(z,\nabla \phi)-\mathcal{F}(z,\nabla \phi)\right\|_{p'(\cdot),Q_T}.
\end{split}
\]
The expression in the brackets is uniformly bounded. The second factor tends
to zero by the dominated convergence theorem because
$\mathcal{F}_{\epsilon}(z,\nabla \phi)-\mathcal{F}(z,\nabla \phi)\to 0$ a.e.
in $Q_T$ and there is the integrable majorant:

\[
|\mathcal{F}_{\epsilon}(z,\nabla \phi)-\mathcal{F}(z,\nabla \phi)|^{p'(z)}\leq C\left(1+|\nabla \phi|^{p(z)}\right)\in L^1(Q_T).
\]
It follows that for every $\phi \in \mathbf{W}_{p(\cdot)}(Q_T)$
\[
0\leq -\int_{Q_T}\left(u_t(u-\phi)+\mathcal{F}(z,\nabla\phi) \cdot \nabla (u-\phi))- f (u-\phi) -\chi(u-\phi) \right)\,dz.
\]
On the other hand, choosing $u-\phi$ for the test-function in \eqref{eq:def-1} for $u^{(\epsilon)}$, we have
\[
0=\int_{Q_T}\left(u_{t}(u-\phi) +\eta\cdot \nabla (u-\phi) - f (u-\phi) - \chi(u-\phi)\right)\,dz.
\]
Summing these relations we find that
\[
\int_{Q_T}\left(\eta-\mathcal{F}(z,\nabla \phi)\right)\cdot \nabla (u-\phi)\,dz\geq 0\qquad \forall \phi\in \mathbf{W}_{p(\cdot)}(Q_T).
\]
Let us choose $\phi=u\pm \lambda \zeta$ with an arbitrary $\zeta\in
\mathbf{W}_{p(\cdot)}(Q_T)$ and a constant $\lambda >0$. Simplifying by
$\lambda$ and then letting $\lambda\to 0^+$ we have
\[
\pm \int_{Q_T}\left(\eta-\mathcal{F}(z,\nabla u)\right)\cdot \nabla \zeta\,dz\geq 0\qquad \forall \zeta\in \mathbf{W}_{p(\cdot)}(Q_T).
\]
It follows that
\[
0=\int_{Q_T}\left(u_{t}\zeta +\mathcal{F}(z,\nabla u)\cdot \nabla \zeta - f \zeta - \chi \zeta\right)\,dz\qquad \forall \zeta\in \mathbf{W}_{p(\cdot)}(Q_T).
\]
Let us identify the limit function $\chi$. Choose $u^{(\epsilon)}-u$ for the test-function in identities \eqref{eq:def-1} for $u^{(\epsilon)}$ and $u$ and combine the results:
\[
\begin{split}
\int_{Q_T} & (\mathcal{F}_{\epsilon}(z,\nabla u^{(\epsilon)})-\mathcal{F}_{\epsilon}(z,\nabla u))\cdot \nabla (u^{(\epsilon)}-u)\,dz =-\int_{Q_T}(u^{(\epsilon)}_{t}-u_t)(u^{(\epsilon)}-u)\,dz
\\
&
+\int_{Q_T}\left(F_{\epsilon}(z,u^{(\epsilon)},\nabla u^{(\epsilon)})-\chi\right) (u^{(\epsilon)}-u) \,dz
\\
&
\qquad +\int_{Q_T} (\mathcal{F}(z,\nabla u)-\mathcal{F}_{\epsilon}(z,\nabla u))\cdot \nabla (u^{(\epsilon)}-u)\,dz \equiv \mathcal{J}_1+\mathcal{J}_2+\mathcal{J}_3.
\end{split}
\]
All terms on the right-hand side of this equality tend to zero as $\epsilon\to 0$. By \eqref{eq:unif-high} and \eqref{eq:conv-1}
\[
\begin{split}
& |\mathcal{J}_1|\leq \left(\|u^{(\epsilon)}_t\|_{2,Q_T}+\|u_t\|_{2,Q_T}\right)\|u^{(\epsilon)}-u\|_{2,Q_T}\to 0,
\\
&
|\mathcal{J}_2|\leq \left(\|F_{\epsilon}(z,u^{(\epsilon)},\nabla u^{(\epsilon)})\|_{2,Q_T}+\|\chi\|_{2,Q_T}\right)\|u^{(\epsilon)}-u\|_{2,Q_T}\leq C\|u^{(\epsilon)}-u\|_{2,Q_T}\to 0.
\end{split}
\]
To estimate $\mathcal{J}_3$ we write
\[
|\mathcal{J}_3|\leq 2\left(\|\nabla u^{(\epsilon)}\|_{p(\cdot),Q_T}+\|\nabla u\|_{p(\cdot),Q_T}\right)\|\mathcal{F}(z,\nabla u)-\mathcal{F}_{\epsilon}(z,\nabla u)\|_{p'(\cdot),Q_T}.
\]
The expression in the brackets is uniformly bounded while the second factor tends to zero by the dominated convergence theorem: $\mathcal{F}(z,\nabla u)-\mathcal{F}_{\epsilon}(z,\nabla u)\to 0$ a.e. in $Q_T$ and there is the integrable majorant
\[
|\mathcal{F}(z,\nabla u)-\mathcal{F}_{\epsilon}(z,\nabla u)|^{p'(z)}\leq C(1+|\nabla u|^{p(z)})\in L^1(Q_T).
\]
It follows that
\[
\int_{Q_T} (\mathcal{F}_{\epsilon}(z,\nabla u^{(\epsilon)})-\mathcal{F}_{\epsilon}(z,\nabla u))\cdot \nabla (u^{(\epsilon)}-u)\,dz \to 0\quad \text{as $\epsilon\to 0$},
\]
whence, by \eqref{eq:mon-reg} and \eqref{eq:R-conv} with $p_m=p$,
\[
\text{$\nabla u^{(\epsilon)}\to \nabla u$ in $L^{p(\cdot)}(Q_T)^N$ and a.e. in $Q_T$}.
\]
The pointwise convergence $u^{(\epsilon)}\to u$, $\nabla u^{(\epsilon)}\to \nabla u$ and the uniform estimates \eqref{eq:unif-high} entail the convergence
\[
\begin{split}
&
(\epsilon^2+|\nabla u^{(\epsilon)}|^2)^{\frac{s-2}{2}} \nabla u^{(\epsilon)} \to |\nabla u|^{s-2}\nabla u \quad \text{in $L^2(Q_T)^N$},
\\
& (\epsilon^2+ (u^{(\epsilon)})^2)^{\frac{q-2}{2}} u^{(\epsilon)} \to |u|^{q-2} u \quad \text{in $L^2(Q_T)$}.
\end{split}
\]
It follows that $F_{\epsilon}(z,u^{(\epsilon)},\nabla u^{(\epsilon)})\to F(z,u,\nabla u)$ in $L^2(Q_T)$. Passing to the limit in \eqref{eq:def-1} for $u^{(\epsilon)}$ as $\epsilon \to 0$ we obtain \eqref{eq:def-1} for the function $u(z)=\lim u^{(\epsilon)}(z)$.
\end{proof}

\begin{proof}[Proof of Theorem \ref{th:reg-degenerate}]
Let $\Theta_{ij}^{(\epsilon)}$ be the functions chosen in  \eqref{eq:high-diff-2}. By Theorem \ref{th:reg-eps}, $\|\Theta_{ij}^{(\epsilon)}\|_{2,Q_T}\leq C$
with an independent of $\epsilon$ constant $C$, and there exists $\Theta_{ij} \in L^2(Q_T)$ such that $\Theta_{ij}^{(\epsilon)} \rightharpoonup \Theta_{ij}$ in $L^2(Q_T)$ (up to a subsequence). Now we make use of the pointwise convergence of the gradients established in the proof of Theorem \ref{th:existence-degenerate}. For every $\psi\in C_0^{\infty}(Q_T)$
\[
\begin{split}
\left(D_j\psi,|\nabla u|^{p-2}D_i u\right)_{2,Q_T} & = \lim_{\epsilon \to 0}\left(D_j\psi,(\epsilon^2+|\nabla u^{(\epsilon)}|^2)^{\frac{p-2}{2}}D_i u^{(\epsilon)}\right)_{2,Q_T}
\\
& =- \lim_{\epsilon \to 0}\left(\psi,D_j\left((\epsilon^2+|\nabla u^{(\epsilon)}|^2)^{\frac{p-2}{2}}D_i u^{(\epsilon)}\right)\right)_{2,Q_T}
\\
& = -(\psi,\Theta_{ij})_{2,Q_T},
\end{split}
\]
which means that $\Theta_{ij}=D_j\left(|\nabla u|^{p-2}D_i u\right)$ in $L^2(Q_T)$ and
\[
\left\|D_j\left(|\nabla u|^{p-2}D_i u\right)\right\|_{2,Q_T}\leq C
\]
with the constant $C$ from Theorem \ref{th:reg-eps}.
\end{proof}

\section{$C^2$-smooth and convex domains: proof of Theorem \ref{th:convex-domain}}
\label{sec:convex-domain} We begin with the case of a bounded convex domain
$\Omega$. There exists a sequence of smooth convex domains $\{\Omega_k\}$,
$\Omega_{k}\Subset \Omega_{k+1}$, $k=1,2,\ldots$, approximating $\Omega$ from
the interior: $\Omega_k\uparrow \Omega$ as $k\to \infty$. To construct such a
sequence take a smooth convex exhaustion of $\Omega$, that is, a smooth,
negative, convex function $\psi$ such that $\lim_{x\to
\partial\Omega}\psi(x)=0$, and choose for $\Omega_k$ the sub-level sets of
$\psi$:
\[
\Omega_k=\left\{x\in \Omega:\,\psi(x)<-{1}/{k}\right\}.
\]
For the existence of an exhaustion function we refer to \cite{Blocki-1997}.

Let $\{u_k\}$ be the sequence of solutions of the problems
\begin{equation}
\label{eq:main-k}
\begin{split}
& u_{k,t}-\operatorname{div}\left(|\nabla u_k|^{p(z)-2}\nabla
u_k\right)=f_k(z)+F(z,u_k,\nabla u_k)\quad \text{in $Q_{T,k}=\Omega_k\times
(0,T)$},
\\
& \text{$u_k=0$ on $\partial\Omega_k\times (0,T)$},
\\
& \text{$u_k(x,0)=u_{0,k}(x)$ in $\Omega$},
\end{split}
\end{equation}
where $u_{0,k}\in W_0^{1,p(\cdot)}(\Omega_k)$, $u_{0,k}=0$ in $\Omega\setminus
\Omega_k$, $u_{0,k}\to u_0$ in $W_0^{1,p(\cdot)}(\Omega)$, and $f_k(z)=f(z)$
in $Q_{T,k}$, $f_k=0$ in $Q_T\setminus Q_{T,k}$. By Theorem
\ref{th:existence-degenerate} problem \eqref{eq:main-k} has a solution $u_k$
satisfying estimates \eqref{eq:unif-est-degenerate}-\eqref{eq:higher-m}.
Define the functions
\[
w_k(z)=\begin{cases} u_k & \text{if $z\in Q_{T,k}$},
\\
0 & \text{if $z\in Q_T\setminus Q_{T,k}$}.
\end{cases}
\]
Since each of $w_k$ satisfies the uniform estimates
\eqref{eq:unif-est-degenerate}-\eqref{eq:higher-m}, the sequence $\{w_k\}$ has
the convergence properties similar to \eqref{eq:conv-2}: there are functions
$w$, $\eta$, $\chi$ such that  (up to a subsequence)
\begin{equation}
\label{eq:conv-3}
\begin{split}
& \text{$w_k\to w$ $\star$-weakly in $L^\infty(0,T;L^2(\Omega))$, strongly in
$L^2(Q_T)$ and a.e. in $Q_T$},
\\
& \text{$w_{k,t}\rightharpoonup w_t$ in $L^2(Q_T)$},\quad \text{$\nabla w_k
\rightharpoonup \nabla w$ in $L^{p(\cdot)}(Q_T)^N$},
\\
& \text{$|\nabla w_k|^{p(z)-2}\nabla w_k \rightharpoonup \eta$ in
$L^{p'(\cdot)}(Q_T)^N$},
\\
& \text{$F(z,w_k,\nabla w_k)\rightharpoonup \chi$ in $L^2(Q_T)$}.
\end{split}
\end{equation}
Moreover, by virtue of \eqref{eq:p-est-3}
\begin{equation}
\label{eq:higher-int-k}
\int_{Q_T}|\nabla w_k|^{2(p-1)+r}\,dz=\int_{\Omega_k\times (0,T)}|\nabla
u_k|^{2(p-1)+r}\,dz\leq C,
\end{equation}
whence
\begin{equation}
\label{eq:higher-convex} \int_{Q_T}|\nabla w|^{2(p-1)+r}\,dz\leq C,\qquad r\in
\left(0,\frac{4}{N+2}\right),
\end{equation}
with a constant $C$ independent of $\partial\Omega$.

Let us check that the limit function $w$ is a strong solution of problem
\eqref{eq:main}. Let $\phi\in C^\infty([0,T];C_0^\infty(\Omega))$.  For all sufficiently large $k \in \mathbb{N}$ we have $\operatorname{supp}\phi\Subset Q_{T,k}$ and

\begin{equation}
\label{eq:def-2} \int_{Q_T}\left(w_{k,t}\phi+|\nabla w_k|^{p-2}\nabla w_k\cdot
\nabla \phi -f_k\phi -F(z,w_k,\nabla w_k)\phi\right)\,dz=0.
\end{equation}
Letting $k\to \infty$ and using \eqref{eq:conv-3} we find that
\begin{equation}
\label{eq:def-3} \int_{Q_T}\left(w_{t}\phi+\eta\cdot \nabla \phi -f\phi
-\chi\phi\right)\,dz=0.
\end{equation}
Since $p(\cdot)$ is Lipschitz-continuous in $Q_T$, the set $C^\infty([0,T];C_0^\infty(\Omega))$ is dense in $\mathbf{W}_{p(\cdot)}(Q_T)$ (see \cite[Sec.4]{Die-Nag-Ruz-2012}), equality \eqref{eq:def-3} holds true for every test-function $\phi\in \mathbf{W}_{p(\cdot)}(Q_T)$.

\begin{lemma}
\label{le:conv-grad}
The sequence $\{\nabla w_k\}$ contains a subsequence that converges to $\nabla w$ a.e. in $Q_T$.
\end{lemma}

\begin{proof}
Take an arbitrary set $\omega\Subset \Omega$, denote $d_\omega=\operatorname{dist}(\omega,\partial\Omega)$ and pick a smooth function $\psi$ such that
\[
0\leq \psi(x)\leq 1, \qquad \psi(x)\equiv \begin{cases} 1 & \text{for $x\in \omega$},
\\
0 & \text{if $\operatorname{dist}(x,\omega)  > \frac{d_\omega}{2}$},
\end{cases}
\qquad
|\nabla \psi |\leq \frac{C}{d_\omega}.
\]
There exists $k_0\in \mathbb{N}$ so large that $\psi(x)=0$ in $\Omega\setminus\Omega_{k_0}$. Let $m,k>k_0$. Combining identities \eqref{eq:def-2} for $w_m$ and $w_k$ with the
test-function $\phi=(w_m-w_k)\psi$ we obtain the equality

\begin{equation}
\label{eq:inter-prim}
\begin{split}
\frac{1}{2} & \|\sqrt{\psi}(w_m-w_k)\|^2_{2,\Omega}(T)+\int_{Q_T}\psi\left(|\nabla
w_m|^{p-2}\nabla w_m-|\nabla w_k|^{p-2}\nabla w_k\right)\cdot \nabla (w_m-w_k)\,dz
\\
& = \frac{1}{2}\|\sqrt{\psi}(w_m-w_k)\|^2_{2,\Omega}(0)+\int_{Q_T}(f_m-f_k)\phi\,dz
+\int_{Q_T}(F(z,w_m,\nabla w_m)- F(z,w_k,\nabla w_k))\phi\,dz
\\
& \quad - \int_{Q_T}(w_m-w_k)\left(|\nabla
w_m|^{p-2}\nabla w_m-|\nabla w_k|^{p-2}\nabla w_k\right)\cdot \nabla \psi \,dz\equiv \sum_{i=1}^4\mathcal{I}_i.
\end{split}
\end{equation}
By the choice of $\{u_{0,k}\}$ and $f_k$, and due to \eqref{eq:conv-3}

\[
|\mathcal{I}_1|+ |\mathcal{I}_2|\leq \frac{1}{2}\|u_{0,m}-u_{0,k}\|_{2,\Omega}^2+\|f_k-f_m\|_{2,Q_T}\|w_m-w_k\|_{2,Q_T}\to 0\quad \text{as $m,k\to \infty$}.
\]
By virtue of estimates \eqref{eq:unif-est-degenerate}-\eqref{eq:higher-m} for the functions $u_m$ in $Q_{T,m}$ and $w_m$ in $Q_T$ and the assumptions on the structure of the equation
$\|F(z,w_m,\nabla w_m)\|_{2,Q_T}\leq C_F$ with an independent of $m$ constant $C_F$. Thus,

\[
|\mathcal{I}_3|\leq 2C_F\|w_k-w_m\|_{2,Q_T}\to 0\quad \text{as $m,k\to \infty$}.
\]
Because of the higher integrability property \eqref{eq:higher-int-k}

\[
\begin{split}
|\mathcal{I}_4| & \leq \dfrac{C}{d_\omega}\left(\||\nabla w_m|^{p-1}\|_{2,Q_T}+\||\nabla w_k|^{p-1}\|_{2,Q_T}\right)\|w_m-w_k\|_{2,Q_T}
\\
& \leq C'\|w_m-w_k\|_{2,Q_T}\to 0\quad \text{as $m,k\to \infty$}.
\end{split}
\]
Dropping the first nonnegative term on the left-hand side of \eqref{eq:inter-prim}, we find that

\[
\begin{split}
0\leq \int_{\omega\times (0,T)} & \left(|\nabla
w_k|^{p-2}\nabla w_k-|\nabla w_m|^{p-2}\nabla w_m\right)\cdot \nabla (w_m-w_k)\,dz\to 0\quad \text{as $m,k\to \infty$}.
\end{split}
\]
We now follow the proof of Lemma \ref{le:strong-conv-grad}. Using an analog of \eqref{eq:R-conv}, we
find that $\|\nabla (w_k-w_m)\|_{p(\cdot),\omega\times (0,T)}\to 0$ as $m,k\to \infty$. It follows
that $\{\nabla w_k\}$ is a Cauchy sequence in $L^{p(\cdot)}(\omega\times (0,T))^N$, and contains a subsequence that converges to $\nabla w$ a.e. in $\omega\times (0,T)$. Taking an expanding sequence of domains $\{\omega_s\}$, $\omega_s\uparrow \Omega$ as $s\to \infty$, and applying the diagonal procedure we extract a subsequence of $\{\nabla w_k\}$ that converges to $\nabla w$ a.e. in $Q_T$.
\end{proof}
By \cite[Ch.1,Lemma 1.3]{Lions-1969} we conclude that $\eta=|\nabla w|^{p-2}\nabla w$ and $\chi=F(z,w,\nabla w)$,
as required. Since $w$ satisfies \eqref{eq:def-3} with every test-function
$\phi\in \mathbf{W}_{p(\cdot)}(Q_T)$, $w$ is a strong solution of problem
\eqref{eq:main}. The inclusion $w\in C^0([0,T];L^2(\Omega))$ follows from the
Aubin-Lions lemma.

To prove the second-order regularity we follow the proof of Theorem
\ref{th:reg-degenerate}. Since $|\nabla u_k|^{p-2}\nabla u_k\in
W^{1,2}(Q_{T,k})$ for every $k$, using the diagonal process we may choose a
subsequence $\{w_k\}$ and a function $\Theta_{ij}$ such that

\[
D_i\left(|\nabla w_k|^{p-2}D_iw_k\right)\rightharpoonup \Theta_{ij} \quad
\text{in $L^2(Q_{T,m})$ for every $m$}.
\]
Take an arbitrary function $\psi\in C_0^\infty(Q_T)$ and fix $k_0\in
\mathbb{N}$ such that $\operatorname{supp}\psi\subset \Omega_{k_0}\times
(0,T)$. Then
\[
\begin{split}
\left(D_j\psi,|\nabla w|^{p-2}D_iw\right)_{2,Q_T} & = \lim_{k\to \infty, k\geq
k_0}\left(D_j\psi,|\nabla w_k|^{p-2}D_iw_k\right)_{2,Q_T}
\\
& = \lim_{k\to \infty, k\geq k_0}\left(D_j\psi,|\nabla
w_k|^{p-2}D_iw_k\right)_{2,Q_{T,k}}
\\
& = - \lim_{k\to \infty, k\geq k_0}\left(\psi,D_j\left(|\nabla
w_k|^{p-2}D_iw_k\right)\right)_{2,Q_{T,k_0}}
\\
& = - (\psi,\Theta_{ij})_{2,Q_{T,k_0}}
\\
& = - (\psi,\Theta_{ij})_{2,Q_{T}}.
\end{split}
\]
Let $\partial\Omega\in C^2$. There exists a sequence $\{\Omega_k\}$,
$\Omega_k\Subset\Omega_{k+1}$, $\Omega_k\uparrow \Omega$, such that
$\partial\Omega_k\in C^\infty$ and the parametrizations of $\partial\Omega_k$
are uniformly bounded in $C^2$, see \cite[Subsec.7.4]{Arora-Shm-RACSAM-2023}
for the choice of $\Omega_k$. Let $\{u_k\}$ be the sequence of solutions to
problems \eqref{eq:main-k} in the cylinders $Q_{T,k}$. The a priori estimates
for $u_k$ depend on the $C^2$-norms of the parametrizations of
$\partial\Omega_k$, which are independent of $k$. The proof is completed as in
the case of the convex domain $\Omega$.

%

\

\noindent On behalf of all authors, the corresponding author states that there is no conflict of interest.


\begin{thebibliography}{10}

\bibitem{Seregin-Acerbi-Mignione-2004}
{\sc E.~Acerbi, G.~Mingione, and G.~A. Seregin}, {\em Regularity results for
  parabolic systems related to a class of non-{N}ewtonian fluids}, Ann. Inst.
  H. Poincar\'{e} C Anal. Non Lin\'{e}aire, 21 (2004), pp.~25--60.

\bibitem{Al-Zhi-2010}
{\sc Y.~Alkhutov and V.~Zhikov}, {\em Existence theorems for solutions of
  parabolic equations with a variable order of nonlinearity}, Tr. Mat. Inst.
  Steklova, 270 (2010), pp.~21--32.

\bibitem{Ant-Shm-2009}
{\sc S.~Antontsev and S.~Shmarev}, {\em Anisotropic parabolic equations with
  variable nonlinearity}, Publ. Mat., 53 (2009), pp.~355--399.

\bibitem{Ant-Shm-AA-2017}
{\sc S.~Antontsev and S.~Shmarev}, {\em Higher regularity of solutions of
  singular parabolic equations with variable nonlinearity}, Applicable
  Analysis, 98 (2019), pp.~310--331.

\bibitem{Ant-Shm-NA-2020}
\leavevmode\vrule height 2pt depth -1.6pt width 23pt, {\em Global estimates for
  solutions of singular parabolic and elliptic equations with variable
  nonlinearity}, Nonlinear Analysis: Theory, Methods and Applications, 195
  (2020), p.~111724.

\bibitem{Arora-Shm-2021}
{\sc R.~Arora and S.~Shmarev}, {\em Strong solutions of evolution equations
  with {$p(x,t)$}-{L}aplacian: existence, global higher integrability of the
  gradients and second-order regularity}, J. Math. Anal. Appl., 493 (2021),
  pp.~124506, 31.

\bibitem{Arora-Shm-ANONA-2023}
\leavevmode\vrule height 2pt depth -1.6pt width 23pt, {\em Double-phase
  parabolic equations with variable growth and nonlinear sources}, Adv.
  Nonlinear Anal., 12 (2023), pp.~304--335.

\bibitem{Arora-Shm-RACSAM-2023}
\leavevmode\vrule height 2pt depth -1.6pt width 23pt, {\em Existence and
  regularity results for a class of parabolic problems with double phase flux
  of variable growth}, Rev. R. Acad. Cienc. Exactas F\'{\i}s. Nat. Ser. A Mat.
  RACSAM, 117 (2023), pp.~Paper No. 34, 48.

\bibitem{BCDM-2023}
{\sc A.~K. Balci, A.~Cianchi, L.~Diening, and V.~Maz'ya}, {\em A pointwise
  differential inequality and second-order regularity for nonlinear elliptic
  systems}, Math. Ann., 383 (2022), pp.~1775--1824.

\bibitem{Blocki-1997}
{\sc Z.~B{\l}ocki}, {\em Smooth exhaustion functions in convex domains},
  Proceedings of the American Mathematical Society, 125 (1997), pp.~477--484.

\bibitem{Bogelein-Duzaar-2011}
{\sc V.~B\"{o}gelein and F.~Duzaar}, {\em H\"{o}lder estimates for parabolic
  {$p(x,t)$}-{L}aplacian systems}, Math. Ann., 354 (2012), pp.~907--938.

\bibitem{Bogelein-2022}
{\sc V.~B\"{o}gelein, F.~Duzaar, N.~Liao, and C.~Scheven}, {\em Gradient
  {H}\"{o}lder regularity for degenerate parabolic systems}, Nonlinear Anal.,
  225 (2022), pp.~Paper No. 113119, 61.

\bibitem{Challal-2011}
{\sc S.~Challal and A.~Lyaghfouri}, {\em Second order regularity for the
  {$p(x)$}-{L}aplace operator}, Math. Nachr., 284 (2011), pp.~1270--1279.

\bibitem{Cianchi-Maz'ya-2018}
{\sc A.~Cianchi and V.~G. Maz'ya}, {\em Second-order two-sided estimates in
  nonlinear elliptic problems}, Arch. Ration. Mech. Anal., 229 (2018),
  pp.~569--599.

\bibitem{Cianchi-Maz'ya-2019}
\leavevmode\vrule height 2pt depth -1.6pt width 23pt, {\em Optimal second-order
  regularity for the {$p$}-{L}aplace system}, J. Math. Pures Appl. (9), 132
  (2019), pp.~41--78.

\bibitem{Cianchi-Maz'ya-2019-1}
\leavevmode\vrule height 2pt depth -1.6pt width 23pt, {\em Second-order
  regularity for parabolic {$p$}-{L}aplace problems}, J. Geom. Anal., 30
  (2020), pp.~1565--1583.

\bibitem{DeFilippis-2020}
{\sc C.~De~Filippis}, {\em Gradient bounds for solutions to irregular parabolic
  equations with {$(p, q)$}-growth}, Calc. Var. Partial Differential Equations,
  59 (2020), pp.~Paper No. 171, 32.

\bibitem{DHHR-2011}
{\sc L.~Diening, P.~Harjulehto, P.~H$\ddot{a}$st$\ddot{o}$, and
  M.~R$\dot{u}$$\check{z}$i$\breve{c}$ka}, {\em Lebesgue and Sobolev Spaces
  with Variable Exponents}, Springer, Berlin, 2011.
\newblock Series: Lecture Notes in Mathematics, Vol. 2017, 1st Edition.

\bibitem{Die-Nag-Ruz-2012}
{\sc L.~Diening, P.~N\"agele, and M.~R$\dot{u}$$\check{z}$i$\breve{c}$ka}, {\em
  Monotone operator theory for unsteady problems in variable exponent spaces},
  Complex Var. Elliptic Equ., 57 (2012), pp.~1209--1231.

\bibitem{Ding-Zhang-Zhou-2020}
{\sc M.~Ding, C.~Zhang, and S.~Zhou}, {\em Global boundedness and {H}\"{o}lder
  regularity of solutions to general {$p(x,t)$}-{L}aplace parabolic equations},
  Math. Methods Appl. Sci., 43 (2020), pp.~5809--5831.

\bibitem{Duzaar-Mignione-Steffen-2011}
{\sc F.~Duzaar, G.~Mingione, and K.~Steffen}, {\em Parabolic systems with
  polynomial growth and regularity}, Mem. Amer. Math. Soc., 214 (2011),
  pp.~x+118.

\bibitem{Feng-Parviainen-Sarsa-2022}
{\sc Y.~Feng, M.~Parviainen, and S.~Sarsa}, {\em On the second-order regularity
  of solutions to the parabolic {$p$}-{L}aplace equation}, J. Evol. Equ., 22
  (2022), pp.~Paper No. 6, 17.

\bibitem{Feng-Parviainen-Sarsa-2023}
\leavevmode\vrule height 2pt depth -1.6pt width 23pt, {\em A systematic
  approach on the second order regularity of solutions to the general parabolic
  {$p$}-{L}aplace equation}, Calc. Var. Partial Differential Equations, 62
  (2023), p.~Paper No. 204.

\bibitem{Grisvard}
{\sc P.~Grisvard}, {\em Elliptic problems in nonsmooth domains}, vol.~69 of
  Classics in Applied Mathematics, Society for Industrial and Applied
  Mathematics (SIAM), Philadelphia, PA, 2011.
\newblock Reprint of the 1985 original [ MR0775683], With a foreword by Susanne
  C. Brenner.

\bibitem{LSU}
{\sc O.~A. Lady\v{z}enskaja, V.~A. Solonnikov, and N.~N. Ural'ceva}, {\em
  Linear and quasilinear equations of parabolic type}, Translations of
  Mathematical Monographs, Vol. 23, American Mathematical Society, Providence,
  R.I., 1968.
\newblock Translated from the Russian by S. Smith.

\bibitem{LU}
{\sc O.~A. Ladyzhenskaya and N.~N. Ural'tseva}, {\em Linear and quasilinear
  elliptic equations}, Academic Press, New York-London, 1968.
\newblock Translated from the Russian by Scripta Technica, Inc, Translation
  editor: Leon Ehrenpreis.

\bibitem{Lions-1969}
{\sc J.-L. Lions}, {\em Quelques m\'{e}thodes de r\'{e}solution des probl\`emes
  aux limites nonlin\'{e}aires}, Dunod, Paris; Gauthier-Villars, Paris, 1969.

\bibitem{McShane-1934}
{\sc E.~J. McShane}, {\em Extension of range of functions}, Bull. Amer. Math.
  Soc., 40 (1934), pp.~837--842.

\bibitem{OK-2018}
{\sc J.~Ok}, {\em Regularity for parabolic equations with time dependent
  growth}, J. Math. Pures Appl. (9), 120 (2018), pp.~253--293.

\bibitem{Scheven-2010}
{\sc C.~Scheven}, {\em Regularity for subquadratic parabolic systems: higher
  integrability and dimension estimates}, Proc. Roy. Soc. Edinburgh Sect. A,
  140 (2010), pp.~1269--1308.

\bibitem{NA-Sh-2018}
{\sc S.~Shmarev}, {\em On the continuity of solutions of the nonhomogeneous
  evolution {$p(x,t)$}-{L}aplace equation}, Nonlinear Anal., 167 (2018),
  pp.~67--84.

\bibitem{SSS-2022}
{\sc S.~Shmarev, J.~Simsen, and M.~S. Simsen}, {\em Stability for evolution
  equations with variable growth}, Mediterr. J. Math., 19 (2022), pp.~Paper No.
  183, 25.

\bibitem{Wang-2023}
{\sc Y.~Wang, C.~Zhang, and Y.~Zhua}, {\em A second-order {S}obolev regularity
  for $p(x)$-laplace equations}, J. Math.Anal.Appl.,  (2023), p.~Paper
  No.127328.

\bibitem{Yao-2015}
{\sc F.~Yao}, {\em H\"{o}lder regularity for the general parabolic
  {$p(x,t)$}-{L}aplacian equations}, NoDEA Nonlinear Differential Equations
  Appl., 22 (2015), pp.~105--119.

\end{thebibliography}
\end{document}